\documentclass[12pt, reqno]{amsart}
\usepackage[utf8]{inputenc}

\usepackage{graphicx}
\usepackage{comment}
\usepackage{fourier}
\usepackage[T1]{fontenc}
\usepackage[margin=0.75in]{geometry}
\usepackage{amsthm}
\usepackage{amsmath}
\usepackage{amssymb}
\usepackage{mathtools}
\usepackage{url}
\usepackage{tabularx}

\mathtoolsset{showonlyrefs}

\newcommand {\Z}{\mathbb{Z}}

\newcommand{\la}{\lambda}
\newcommand{\re}{\operatorname{Re}}

\newcommand{\sumast}{\mathop{\sum\nolimits^{\mathrlap{\ast}}}}

\newcommand{\Mod}[1]{\ (\mathrm{mod}\ #1)}
\newtheorem{thm}{Theorem}[section]
\newtheorem{defn}[thm]{Definition}
\newtheorem{lemma}[thm]{Lemma}
\newtheorem{prop}[thm]{Proposition}

\newtheorem{remark}[thm]{Remark}
\newtheoremstyle{named}{}{}{\itshape}{}{\bfseries}{.}{.5em}{\thmnote{#3}}
\theoremstyle{named}

\numberwithin{equation}{section}

\usepackage{multicol}
\usepackage{xcolor}
\usepackage{hyperref}
\hypersetup{
    colorlinks,
    linkcolor={red!50!black},
    citecolor={blue!50!black},
    urlcolor={blue!80!black}
}

\newcommand{\mz}{\ensuremath{\mathbb Z}}
\newcommand{\mr}{\ensuremath{\mathbb R}}

\newcommand{\intR}{\int_{-\infty}^{\infty}}

\newcommand{\shortmod}{\ensuremath{\negthickspace \negthickspace \negthickspace \pmod}}
\newcommand{\sumstar}{\sideset{}{^*}\sum}

\title{Level aspect subconvexity for $\textrm{GL(2)}\times \textrm{GL(2)}$ $\textrm{L}$-functions}

\author[K. Aggarwal]{Keshav Aggarwal}
\address{Department of Mathematics \\ Indian Institute of Technology Bombay\\ Mumbai\\ India}
\email{keshav@math.iitb.ac.in}

\author[S. Kumar] {Sumit Kumar}
 \address{Department of Mathematics \\
 	  Aarhus University \\
 	  8000 Aarhus C \\
 		Denmark}		
 \email{sumit.kumar@math.au.dk}

 \author[C.-H. Kwan]{Chung-Hang Kwan}
 \address{Department of Mathematics \\
 	  University College London (UCL) \\
    25 Gordon Street \\
 	  London WC1H 0AY \\
 	UK}		
 \email{ucahckw@ucl.ac.uk}

 \author[W. H. Leung]{Wing Hong Leung}
 \address{Department of Mathematics \\
 	  Rutgers University \\
 	  Piscataway \\
 	  NJ 08854 \\
 		U.S.A.}		
 \email{joseph.leung@rutgers.edu}

 \author[J. Li]{Junxian Li}
 \address{Department of Mathematics\\
 University of California, Davis\\
 1 Shields Avenue\\ Davis\\ CA 95616\\ U. S. A
 }
 \email{junxian@math.ucdavis.edu}

\author[M. Young]{Matthew P. Young}
 \address{Department of Mathematics \\
 	  Rutgers University \\
 	  Piscataway \\
 	  NJ 08854 \\
 		U.S.A.}		
 \email{mpy4@math.rutgers.edu}

\date{}

\begin{document}

\thanks{This material is based upon work supported by the National Science Foundation under agreement No. DMS-2302210 (M.Y.) and the EPSRC grant: EP/W009838/1 (C.H.K.). K.A. thanks Jasmin Matz for support via the Carlsberg Research Grant: Density and Approximation in Automorphic Families.  Any opinions, findings and conclusions or recommendations expressed in this material are those of the authors and do not necessarily reflect the views of the National Science Foundation. }

\begin{abstract}
    Let $f$ be a newform of prime level $p$ with any central character $\chi (\bmod\, p)$, and let $g$ be a fixed cusp form or Eisenstein series for $\hbox{SL}_{2}(\Z)$. We prove the subconvexity bound: for any $\varepsilon>0$,
\begin{align*}
L(1/2, f \otimes g) \ll p^{1/2-1/524+\varepsilon},
\end{align*}
where the implied constant depends on $g$, $\varepsilon$, and the archimedean parameter of $f$. This improves upon the previously best-known result by Harcos and Michel.  Our method ultimately relies on non-trivial bounds for bilinear forms in Kloosterman fractions pioneered by Duke, Friedlander, and Iwaniec, with later innovations by Bettin and Chandee.  
\end{abstract}

\subjclass[]{}
\keywords{}

\maketitle
\section{Introduction}
\subsection{Discussion of results}

There are a handful of methods to prove subconvexity bounds for $L$-functions. Among these, the moment method has been widely used for many decades. This method embeds a given $L$-function into a family and analyzes a suitable (e.g., amplified) moment averaged over the family;  see \cite{DFI01, DFI02}, for instance.
The moment method has been largely constrained to $L$-functions for $\hbox{GL}_1$ and $\hbox{GL}_2$, though with some notable exceptions 
such as \cite{Xiaoqing, Blomer2015OnTS,  Nelson}. 

Around 2010, Munshi pioneered the delta method for subconvexity, which dispenses with averaging over a family.
This approach was the first to address subconvexity for a wide range of $\hbox{GL}_3$ $L$-functions across various aspects, such as in \cite{MunshiIII, MunshiIV}. Moreover, the delta method has yielded strong subconvexity bounds, exemplified by results in \cite{LinJEMS, Aggarwal}. 
These notable successes over the past decade have spurred further explorations and new applications of this method.

Nevertheless, some of the most difficult subconvexity problems have only been resolved with the moment method.  A prominent example is the subconvexity for $\hbox{GL}_{2}\times \hbox{GL}_{2}$ Rankin--Selberg $L$-functions in the \emph{level aspect}, which has important applications to equidistribution problems of Heegner points on Shimura curves (e.g., \cite{Mi04}, \cite{HarcosMichel}). One of the most difficult cases is when the level is a prime $p$, in which case the problem is stated as follows.
Let $f$ be a cuspidal Hecke newform of prime level $p$ with central character $\chi \, (\bmod\, p)$ and its archimedean parameter (i.e., weight or Laplace eigenvalue) treated as fixed, and let $g$ be a fixed Hecke newform of level $1$ (say).  The problem is to establish a bound of the form
\begin{equation}\label{eq:subc}
    L(1/2, f \otimes g) \ \ll \ p^{1/2-\delta+\varepsilon}
\end{equation}
for some absolute constant $\delta>0$, and where the implied constant depends on $\varepsilon > 0$. When $g$ is an Eisenstein series, \eqref{eq:subc} becomes
\begin{equation*}
    L(1/2, f) \ \ll \ p^{1/4-\delta/2 + \varepsilon/2},
\end{equation*}
which was first proved by Duke, Friedlander, and Iwaniec in \cite{DFI2} for $\chi$ trivial, and \cite{DFI01} for $\chi$ primitive.
For the case where both of the forms $f$ and $g$ are cuspidal, 
Kowalski, Michel, and VanderKam \cite{KMV02} proved a subconvex bound for $L(1/2, f \otimes g)$ when  $\chi$ is trivial 
and $f$ is holomorphic. Subsequently,  Michel \cite{Mi04} addressed the case where $\chi$ is primitive and $g$ is holomorphic, whereas Harcos and Michel \cite{HarcosMichel} treated the case where $\chi$ is primitive and $g$ is a Maass form. We should also mention that Michel and Venkatesh \cite{MV} famously proved subconvexity for $\hbox{GL}_2 \times \hbox{GL}_2$ $L$-functions with arbitrary fixed $g$ and with $f$ varying in any aspect. The numerical values of $\delta$'s obtained in these works are recorded in Table \ref{table:subconvexity} in Section \ref{Otherwo}.

For the subconvexity problem \eqref{eq:subc}, the delta method has thus far been applied only to the case where both $f$ and $g$ are cuspidal and $\chi$ is trivial by  Raju \cite{Raju}. However, his method does not work when  $g$ is an Eisenstein series (\cite[p. 5]{Raju}). In fact, as commented in Kumar, Munshi, and  Singh  \cite{KMS24},   ``(the) effectivity and adaptability (of the delta method) in the more
arithmetic problem of level aspect remains a point of deliberation. In particular, it
seems that new inputs are required.''

 In this paper, we employ the delta method to prove a 
 subconvexity bound for $L(1/2, f \otimes g)$ that holds \emph{uniformly} for $f$ holomorphic or Maass, $g$ holomorphic, Maass or Eisenstein, and  $\chi$ primitive or trivial, in contrast to previous works using the moment method.  
 Our approach also overcomes the existing limitations of the delta method. 
 Furthermore, when $f$ and $g$ are both cuspidal, our bound improves upon the previous best-known result by Harcos \cite{Harcos}, where \eqref{eq:subc} was obtained  with $\delta = 1/1413$.  Our main result is the following.

\begin{thm}
\label{thm:mainthm}
    Let $p$ be prime, and let $f$ be a Hecke newform of level $p$ with central character $\chi\, (\bmod\, p)$.  Let $g$ be a Hecke eigenform of level $1$. Then we have
    \begin{equation}\label{eq:mainbdd}
        L(1/2, f \otimes g) \, \ll \, p^{1/2-1/524 + \varepsilon},
    \end{equation}
    where the implied constant depends on $g$, $\varepsilon$, and the archimedean parameter of $f$.
\end{thm}

\begin{remark}
  The proof easily generalizes to show the same bound \eqref{eq:mainbdd} for $L(1/2+it, f \otimes g)$, with polynomial dependence on $t$ and on the archimedean parameter of $f$.
\end{remark}

\begin{remark}
    The bound \eqref{eq:mainbdd} has the pleasant feature that it is \emph{independent} of\,  $\theta$, i.e., the approximation towards the Ramanujan conjecture. 
\end{remark}

\begin{remark}
\label{remark:conditionalamplifier}
    Under the assumption
    \begin{align}
    \label{eq:conditionalamplifier}
\mathcal{A}_1(L) \, := \, \sum_{\substack{L/2< \nu \leq L\\ \nu \text{ prime}}}|\lambda_f(\nu)|^2 \, \gg_{} \, L^{1-\varepsilon},
    \end{align}
the bound \eqref{eq:mainbdd} in Theorem \ref{thm:mainthm} can be improved to 
    \begin{align}
    \label{eq:RemarkTheoremConditionalAmplifier}
        L(1/2, f \otimes g) \, \ll \, p^{1/2-1/302 + \varepsilon};
    \end{align}
    see the discussion following \eqref{final bound}.
This observation is helpful for comparing the main argument in Section \ref{sect:proof} with the sketch presented in Section 
\ref{section:sketch}, where 
\eqref{eq:conditionalamplifier} is assumed for ease of exposition. 

\end{remark}
\begin{remark}
In Theorem \ref{thm:mainthm}, the assumption that $p$ is prime is largely to simplify our exposition, and it is expected that the same method works for general levels and arbitrary central characters. 
\end{remark}
\begin{remark}
Preliminary calculations suggest that our method of proof can be adapted to other aspects, such as subconvexity for $L(1/2+it_f, f \otimes g)$ as $t_{f}\to \infty$, where $t_f$ denotes the spectral parameter of $f$. We plan to address this question on another occasion. 

\end{remark}

Unlike the aforementioned works, our proof does not require solving any shifted convolution problems for Fourier coefficients or divisor functions, which is the main reason our bound is independent of $\theta$.
Instead, it makes crucial use of the following theorem by Duke, Friedlander, and Iwaniec \cite[Thm. 2]{DFI97} on \emph{Kloosterman fractions}. 

\begin{thm}[Duke--Friedlander--Iwaniec]
\label{thm:DFI}
Let $\alpha_m$ and $\beta_n$ be arbitrary complex numbers
supported on $M< m \leq 2M$ and $N < n \leq 2N$, respectively.
Let $a$ be a non-zero integer.
Then
\begin{equation}
\label{eq:DFI}
\sum_{(m,n)=1} \alpha_m \beta_n e\Big(\frac{a \overline{m}}{n}\Big)
\ \ll_{} \  \| \alpha \| \| \beta \| (|a| + MN)^{3/8} (M+N)^{11/48 + \varepsilon},
    \end{equation} 
    where\, $||\alpha||:= (\,\sum_{m} |\alpha_{m}|^2)^{1/2}$ and\, $||\beta||:= (\,\sum_{n} |\beta_{n}|^2)^{1/2}$. 
\end{thm}

The work \cite{DFI97} does not involve the spectral theory of automorphic forms, and neither does our method. In other words, our argument relies solely on $\hbox{GL}_1$ harmonic analysis. This is more than just a curiosity, because in particular it means that our method works equally well for holomorphic cusp forms of small weights, including weight $1$. The $L$-functions associated with cusp forms of weight $1$ are of great arithmetic significance, since they include examples such as the class group $L$-functions of imaginary quadratic fields and Artin $L$-functions of degree-two Galois representations. The subconvexity of such $L$-functions has important consequences for the arithmetic statistics of class groups (see \cite[Thm. 2.7--2.8]{DFI02}),  but this problem is widely recognized as technically challenging, as the previous approaches entail deep use of spectral machinery with various surprising phenomena; see, e.g., \cite[pp. 491--492]{DFI02}. In fact, even the full Kuznetsov formula does not appear to have been explicitly written down for this setting.

Our work introduces a new usage of Theorem \ref{thm:DFI} for subconvexity problems. Previously, Theorem \ref{thm:DFI} has been applied in only a few instances of subconvexity \cite{DFI01, DFI02, MunshiHybrid}, all of which also relied on the spectral toolkit for $\mathrm{GL}_2$. More recently, Theorem \ref{thm:DFI} has been refined further by Bettin and Chandee \cite{BC}.
They improved the amount of savings in the bilinear sum
\eqref{eq:DFI}, and also extracted additional savings when averaging over a third variable. More precisely, they showed:
\begin{thm}[Bettin--Chandee] 
\label{thm:BC}
Let $\alpha_m$, $\beta_n$, and $\gamma_k$ be arbitrary complex numbers
supported on $M< m \leq 2M$, $N < n \leq 2N$, and $K\leq k \leq 2K$, respectively.
Let $a$ be a non-zero integer.
Then
\begin{align}
\sum_{k}\sum_{(m,n)=1} \alpha_m \beta_n \gamma_k e\Big(\frac{a k\overline{m}}{n}\Big)
\ &\ll_{} \  \| \alpha \| \| \beta \|\|\gamma\| \Big(1+\frac{|a|K}{MN}\Big)^{1/2} \nonumber
\\&\quad \times  \Big((KMN)^{7/20+\varepsilon}(M+N)^{1/4}+(KMN)^{3/8+\varepsilon} (KM+KN)^{1/8 }\Big),
    \end{align} 
    where  $||\alpha||:= (\,\sum_{m} |\alpha_{m}|^2)^{1/2}$, \, $||\beta||:= (\,\sum_{n} |\beta_{n}|^2)^{1/2}$, and \, $||\gamma||:= (\,\sum_{k} |\gamma_{k}|^2)^{1/2}$.
\end{thm}
We shall use Theorem \ref{thm:BC} instead of Theorem \ref{thm:DFI} to obtain the exponent in \eqref{thm:mainthm}.

\subsection{Comparison with other works}\label{Otherwo}
Table \ref{table:subconvexity} below lists some subconvexity bounds for Rankin--Selberg $L$-functions $L(1/2, f \otimes g)$, where $f$ is cuspidal and varies with level $p\to \infty$ (and $\chi \,(\bmod\, p)$ being the nebentypus), and $g$ is fixed with level $1$. In some cases, $g$ is an Eisenstein series, in which case the bound applies to $L(1/2, f \otimes g) = L(1/2, f)^2$. 
Many of these bounds depend on $\theta$, sometimes in complicated ways. Since not all authors have claimed bounds with explicit dependence on $\theta$, 
for consistency in comparison, we have substituted $\theta = 7/64$ (due to Kim and Sarnak). 

In the following, ``Hol.'', ``Maass'', and ``Eis.'' are abbreviations for holomorphic cusp forms, Maass cusp forms (of weight $0$), and Eisenstein series, respectively.  Also, ``prim.'', ``triv.'', and ``wt.'' stand for ``primitive', ``trivial'', and ``weight'' (for holomorphic cusp forms), respectively.

\begin{table}[h]
\caption{Table of subconvexity bounds}
\label{table:subconvexity}
\begin{tabular}{c|c|c|c|c}
$f$ & $\chi$  & $g$  & $\delta$ as in \eqref{eq:subc} &  citation \\
\hline
Hol. $+$ even wt. &  prim. & Eis. & $ 1/96$  & \cite{DFI01} \\
Hol./Maass & prim. & Eis. & $ 1/23041$ &  \cite{DFI02} \\
Hol. $+$ even wt. & triv. & Hol./Maass  & $ 1/80$ &  \cite{KMV02} \\ 
Hol./Maass & prim. & Hol.  & $ 1/1057$ & \cite{Mi04} \\ 
Hol./Maass & prim./triv.  & Hol./Maass/Eis. & $1/2648$ & \cite{HarcosMichel} \\
Hol./Maass & prim./triv. & Hol./Maass/Eis. & $1/1413$ & \cite{Harcos} \\
Hol./Maass & prim./triv. & Hol./Maass/Eis. & not computed &  \cite{MV} \\
 Hol.$+$even wt./Maass & prim. & Eis. & $1/64$ &  \cite{BlomerKhan} \\
Hol./Maass & triv. & Hol./Maass & $ 25/962$ &  \cite{Raju} \\
Hol./Maass & triv. & Hol./Maass/Eis. & $ 25/384$ &  \cite{Zacharias}
\end{tabular}
\end{table}

\subsection{Acknowledgements}
This work was initiated at an AIM workshop, Delta Symbols and the Subconvexity Problem, held October 16-20, 2023.  It is a pleasure to thank the AIM staff and organizers for the excellent working environment.  We thank Roman Holowinsky and Ritabrata Munshi for their encouragement, and also thank Gergely Harcos for his helpful remarks.

\subsection{Notations} Throughout the article, we use the following standard conventions in analytic number theory. We use $\varepsilon$ to denote an arbitrarily small positive constant that may differ from line to line. 
We write $A\asymp B$ if $A\ll B$ and $B\ll A$. A quantity $A$ is said to be ``very small'' if $A\ll_M p^{-M}$ for any $M>0$. We suppress the dependence on $\varepsilon$, $g$ and the weight/Laplace eigenvalue of $f$ in the asymptotic notations. 

\section{Sketch of proof}
\label{section:sketch}
 In this section, we sketch the proof of Theorem \ref{thm:mainthm}. This sketch is not intended to be fully rigorous but rather to provide a conceptual road map for the proof. After applying the approximate functional equation, the task of proving a subconvex bound for $L(1/2, f\otimes g)$ reduces to establishing the estimate
\begin{equation}
	S(N) := \sum_{n\asymp N} \lambda_f(n) \lambda_g(n) \ll \sqrt{N}p^{1/2-\delta}
\end{equation}
for some $\delta>0$,
uniformly for $N \ll p^{1+\varepsilon}$.

\subsection{Amplification}
For our method, it is essential to apply an \emph{amplification} to $S(N)$ as follows. Let $L\geq1$ be a parameter at our disposal, and $\gamma_{\ell}$ be $1$ if $\ell$ is prime with $\ell \asymp L$ and $0$ otherwise. Then 
\begin{align}\label{temp}
    S(N)\approx \frac{1}{\mathcal{A}_1(L)}\sum_{\ell \asymp L}\gamma_{\ell} \overline{\lambda_f(\ell)}\sum_{n\asymp N}\lambda_f(n \ell )\lambda_g(n), \hspace{15pt} \text{with} \hspace{10pt} \mathcal{A}_1(L) := \sum_{\ell \asymp L} \gamma_{\ell } |\lambda_f(\ell )|^2. 
\end{align}
To simplify our sketch, we assume the lower bound \eqref{eq:conditionalamplifier} for $\mathcal{A}_{1}(L)$, which is expected to be true though currently unproven. 
This is a pervasive, but well-known technical issue with amplifiers. A standard work-around, first presented in \cite{DFI2}, exploits the fact that $\lambda_f(\ell)$ and $\lambda_f(\ell^2)$ cannot be simultaneously small.
See Section \ref{sect.Amp} for our precise choice of amplifier.

\subsection{Delta method}
We use the DFI delta method (Lemma \ref{lemma:DeltaSymbol})
with $C \asymp \sqrt{NL}$,
giving 
\begin{align}
		S(N)&\approx\frac{1}{\mathcal A_1(L)} \sum_{\ell \asymp L}\gamma_{\ell} \overline{\lambda_f(\ell)}\sum_{m\asymp NL}\lambda_f(m)\sum_{n\asymp N}\lambda_g(n)\delta(m=n\ell)\\&\approx \  
	\frac{1}{C\mathcal A_1(L)} \sum_{\ell \asymp L}\gamma_\ell\overline{\lambda_f(\ell)} 
	\sum_{c\ll C} \frac{1}{c}
	\quad \sumstar_{h \shortmod{c}}
	\sum_{m\asymp NL} \lambda_f(m)	e_c(hm)\sum_{n\asymp N} \lambda_g(n)
	e_c(-hn\ell ).
\end{align}

\subsection{Voronoi and Cauchy--Schwarz}

Applying the Voronoi formula (Lemma \ref{prop:voronoi}) to the $m$ and $n$-sums in the last expression, we obtain, in the generic case when $N\asymp p$, $c\asymp C$, the following:
\begin{align}\label{spsum}
	S(p)
	\asymp
	\frac{1}{L^{3/2}} \sum_{\ell\asymp L}\gamma_\ell\overline{\lambda_f(\ell)} \sum_{c\asymp \sqrt{Lp}} \frac{\chi(c)}{c}
		\sum_{m \ll p} \overline{\lambda_f(m)}
	\sum_{n \ll  L} \lambda_g(n)  S(0,pn-{\ell}m;c),
\end{align}
where $S(0,pn-\ell m;c)$ is the Ramanujan sum. A trivial estimate using $\lambda_{f}, \lambda_{g}, \gamma_{\ell} \ll 1$ gives $S(p)\ll \sqrt{L}p$, which is away from the ``boundary'' by a factor of $\sqrt{L}$. 

In this sketch, we consider only the terms with $(c, \ell) = 1$; the contribution from $(c,\ell) > 1$ can be estimated trivially after applying the Voronoi formula; see Lemma \ref{lemma:S0NAbound}.
Also, the contribution from $p\mid m$ in \eqref{spsum} can be treated easily; see Lemma \ref{lemma:trivialbound}.  The terms with $p \mid m$ differ from those with $(p,m) = 1$, since the former case allows for  $pn-\ell m =0$, causing $S(0,pn-\ell m;c)$ to be large for all $c$.

Moving the $m$-sum to the outermost position and applying the Cauchy--Schwarz inequality to the sum in \eqref{spsum}, we observe that 
  $  |S(p)|^2\ll L^{-4} T'$,
where 
\begin{align}
    T'\approx &\sum_{\ell_1, \ell_2\asymp L} \gamma_{\ell_1}\gamma_{\ell_2}\overline{\lambda_f(\ell_1)}\lambda_f(\ell_2) \sum_{n_1,n_2\ll L} \lambda_g(n_1) \lambda_g(n_2)\sum_{c_1, c_2\asymp \sqrt{Lp}} \chi(c_1\overline{c_2}) \nonumber\\
     & \times \sum_{\substack{ m\ll p \\ (m,p) = 1}} S(0,pn_1-\ell_1 m;c_1)S(0,pn_2-\ell_2m;c_2)\label{tsum}.
\end{align}
Write $T' = T'' + T^{(0)}$, where $T''$ and $T^{(0)}$ consists of the terms with $\ell_2n_1\neq \ell_1n_2$ and $\ell_2 n_1 = \ell_1 n_2$ respectively.  A trivial bound is sufficient to estimate $T^{(0)}$ satisfactorily; see Lemma 
\ref{lemma:Tosc0bound}.

\subsection{Poisson summation} 

This step, together with the next, represents the major innovation of this article. We exploit the cancellation in the correlation sum of Ramanujan sums in \eqref{tsum}, an arithmetic structure not present in previous works on the subconvexity of $L(1/2, f\otimes g)$. Most importantly, this structure is \emph{independent} of the choices $\hbox{GL}_{2}$ automorphic forms $f$ and $g$. This is the primary reason why we are able to achieve strong uniformity in our subconvex bound and successfully overcome the limitations of delta methods discussed earlier.

We apply Poisson summation to analyze $T''$. To do so, we first extend the $m$-sum in  $T''$ to include all $m \in \mz$, then subtract back the terms with $p \mid m$, say $T'' = T - T^{(00)}$ accordingly.  Again, a trivial estimation to  $T^{(00)}$ is adequate; see Lemma \ref{lemma:Tosc00bound}. Now, Poisson summation for the sum over $m\asymp p$ in $T$ gives
\begin{align}
    T\approx& \ p\mathop{\sum_{\ell_1, \ell_2\asymp L}\sum_{n_1,n_2\ll L}}_{\ell_2n_1\neq \ell_1n_2}\gamma_{\ell_1}\gamma_{\ell_2}\overline{\lambda_f(\ell_1)}\lambda_f(\ell_2) \lambda_g(n_1) \lambda_g(n_2)\sum_{c_1, c_2\asymp \sqrt{Lp}} \chi(c_1\overline{c_2}) \nonumber\\
     & \ \times \sum_{|m|\ll L}\mathop{\sumast_{h_1\Mod{c_1}}\sumast_{h_2\Mod{c_2}}}_{c_2h_1\ell_1+c_1h_2\ell_2\equiv m \Mod{c_1c_2}}
     e_{c_1}(h_1 n_1 p + h_2 n_2 p).
\end{align}

We first analyze  the contribution of the zero frequency  ($m=0$) to $T$. Here, we essentially have $c_1=c_2$, and the constraint in the $h_1,h_2$-sums eventually reduces to $\ell_2n_1\equiv \ell_1n_2\Mod{c_1}$.  The sum over $c_1$ is controlled by the number of divisors of $\ell_2 n_1 - \ell_1 n_2$. So,
this contribution to $T$ is $\ll L^{9/2}p^{3/2}$ (see Lemma \ref{lemma:T0'bound}), which is satisfactory when $L$ is small enough.

For the non-zero frequencies ($m\neq 0$), the character sums over $h_{1}, h_{2}$ can be evaluated in closed forms; see Lemma \ref{lemma:ShatEvaluation}. In the generic case when $m\asymp L$ and $(c_1,c_2)=(c_1,\ell_1)=(c_2,\ell_2)=1$,
we arrive at
\begin{align}
\label{befadd}
  T \approx \ &p\mathop{\sum_{\ell_1, \ell_2\asymp L}\sum_{n_1,n_2\ll L}}_{\ell_2n_1\neq \ell_1n_2} \gamma_{\ell_1}\gamma_{\ell_2}\overline{\lambda_f(\ell_1)}\lambda_f(\ell_2)  \lambda_g(n_1) \lambda_g(n_2) \sum_{c_1, c_2\asymp \sqrt{Lp}} \chi(c_1\overline{c_2}) \nonumber\\
     &\times \sum_{m \approx L}
      e_{c_1}\left(\overline{c_2 \ell_1} m n_1 p\right)
      e_{c_2}\left(\overline{c_1 \ell_2} m n_2 p\right).
\end{align}
By using the additive reciprocity to 
$e_{c_1}(\overline{c_2\ell_1}mn_1p)$ and observing that $e_{c_1 c_2 \ell}(m n_1 p)$ can be absorbed into the weight function as $mn_1p\asymp c_1c_2\ell_1$, we have roughly 
\begin{align}
   T\approx  &\ p\mathop{\sum_{\ell_1, \ell_2\asymp L} \sum_{n_1,n_2\ll L}}_{\ell_2n_1\neq \ell_1 n_2}\gamma_{\ell_1}\gamma_{\ell_2}\overline{\lambda_f(\ell_1)}\lambda_f(\ell_2)  \lambda_g(n_1) \lambda_g(n_2)\sum_{c_1, c_2\asymp \sqrt{Lp}} \chi(c_1\overline{c_2}) \nonumber\\
     & \ \times \sum_{m\asymp L}e_{c_2\ell_1}(\overline{c_1\ell_2}mp(n_2\ell_1-n_1\ell_2)).
\end{align}

\subsection{Cancellation in Kloosterman fractions}
With $A := c_1 \ell_2$ and $B := c_2 \ell_1$, we arrange $T$ as follows:
\begin{align}
   T\approx &\   p\mathop{\sum_{\ell_1, \ell_2\asymp L} \sum_{n_1,n_2\ll L}}_{\ell_2n_1\neq \ell_1 n_2} \gamma_{\ell_1}\gamma_{\ell_2}\overline{\lambda_f(\ell_1)}\lambda_f(\ell_2) \lambda_g(n_1) \lambda_g(n_2)  
   \\
   &\times\sum_{m\asymp L} \, \bigg(\sum_{\substack{A, B \asymp L^{3/2} p^{1/2} \\ \ell_2 \mid A, \thinspace \ell_1 \mid B}} 
   \chi(A/\ell_2) \overline{\chi}(B/\ell_1)
      e_{B}(\overline{A}mp(n_2\ell_1-n_1\ell_2))\bigg).
\end{align} 

If one applies Theorem \ref{thm:DFI} to the \emph{bilinear form} over $A, B$, with the following choices of parameters: $|a|= mp \cdot  |n_2 \ell_1 - n_1 \ell_2| \ll L^3 p$, $M = N = L^{3/2} p^{1/2}$, and $\alpha_A$ supported on integers divisible by $\ell_2$ (and likewise, $\beta_B$ supported on multiples of $\ell_1$), then we have
\begin{align}\label{mneq0est}
    T \ll L^7p^2(L^{3/2}\sqrt{p})^{-1/48}.
\end{align}
However, \eqref{mneq0est} can be improved by invoking Theorem \ref{thm:BC} for the \emph{trilinear} form over $m, A, B$:
$$
T\ll 
L^7p^2L^{-3/20}(L^{3/2}\sqrt{p})^{-1/20}=L^7p^2 (L^{9/2}\sqrt{p})^{-1/20};
$$
see Lemma \ref{lemma:Tneq0''bound} for the precise result.

In total, we have
\begin{align}
    |S(p)|^2  \ll 
    L^{-1} p^2 + L^3p^2 (L^{9/2}\sqrt{p})^{-1/20} + (\text{lower order terms}).
\end{align}
Choosing 
$L = p^{1/151}$, 
it follows that
\begin{align}
    S(p)\ll p^{ 301/302}, \hspace{20pt} \text{ and } \hspace{20pt} L(1/2, f \otimes g) \ll p^{1/2-1/302},
\end{align}
which is precisely \eqref{eq:RemarkTheoremConditionalAmplifier} under assumption \eqref{eq:conditionalamplifier}.

\section{Preliminary results}\label{sect:prelim}
In this section, we list the standard material that will be used in the proof.

\subsection{Delta symbol}
The following is a variant of the Duke–Friedlander–Iwaniec
delta method (cf. \cite{DFI94}) with a simpler weight function.
\begin{lemma}(\cite[Lemma 3.1]{leung2024shifted})
\label{lemma:DeltaSymbol}
Let $n \in \mathbb{Z}$ be such that $|n| \ll N$, and let $C > N^{\varepsilon}$. Let $U \in C_{c}^{\infty}(\mathbb{R})$ and $W \in C_{c}^{\infty}([-2, -1] \cup [1, 2])$ be non-negative even functions such that $U(x) = 1$ for $x \in [-2, 2]$. Then  
\begin{equation*}
\delta(n = 0) = \frac{1}{C} \sum_{c = 1}^{\infty}\sum_{d = 1}^{\infty} \frac{1}{cd} 
S(0,n;c) F\left(\frac{cd}{C}, \frac{n}{cdC} \right),
\end{equation*}
where $F$ 
is supported on $|x| + |y| \ll 1$, satisfies $F^{(i,j)}(x,y) \ll_{i,j} 1$,
and is
defined by
\begin{equation}
\label{eq:Fdef}
F(x, y) \coloneqq C\left(\sum_{c = 1}^{\infty} W \left(\frac{c}{C} \right)\right)^{-1}\left(W(x) U(x) U(y)-W(y) U(x) U(y)\right).
\end{equation}
\end{lemma}

 We write $F(x,y) = \sum_{i=1}^{2} F_i(x) U_i(y)$ according to \eqref{eq:Fdef} where $F_i, U_i\in C_c^\infty([-2,-1]\cup[1,2])$. 

\subsection{Automorphic data}
For $f\in S_{k}^{*}(p, \chi)$ and $g\in S_{k'}(1)$, we have the (normalized) Fourier expansions
\begin{align}
    f(z) \ = \  \sum_{n\ge 1} \ \lambda_{f}(n)n^{\frac{k-1}{2}} e(nz) \hspace{10pt} \text{ and } \hspace{10pt}   g(z) \ = \  \sum_{n\ge 1} \ \lambda_{g}(n)n^{\frac{k'-1}{2}} e(nz) \hspace{15pt} (z \ \in \ \mathbb{H}). 
\end{align}
Let $\eta_f$ be the eigenvalue of the complex conjugate of the Fricke involution (see \cite[Prop. 14.14]{IK04}). For $\re s \gg  1$, we define
\begin{align}
    L(s, f\otimes g) \ := \ L(2s, \chi) \sum_{n=1}^{\infty} \ \frac{\lambda_{f}(n)\lambda_{g}(n)}{n^{s}}.
\end{align}
It admits an entire continuation and satisfies the functional equation of the form
\begin{align}\label{eq:FE}
    \Lambda(s, f\otimes g)  \ :=  \ \left(\frac{p}{4\pi^2}\right)^{s} \Gamma\left(s + \frac{|k-k'|}{2}\right) \Gamma\left(s+ \frac{k+k'}{2}-1\right) L(s, f\otimes g) \ = \ \eta_{f}^2 \Lambda(1-s, \overline{f} \otimes g);
\end{align}
see \cite[Chp. 5.11]{IK04}.
Note that the arithmetic conductor of $L(s, f\otimes g)$ is $p^2$.  
Let  $w_N\in C_{c}^{\infty}[N/2,2N]$ which satisfies $w_N^{(j)}(x) \ll N^{-j}$. By the approximate functional equation (\cite[Thm. 5.3]{IK04}) and a smooth partition of unity, we have 
\begin{align}
\label{eq:AFEdyadicVersion}
    L(1/2, f\otimes g) \ \ll \ p^{\varepsilon} \, \sum_{\substack{N  \, \text{ dyadic} \\ N \ll p^{1+\varepsilon}}} \  
    \frac{|S(N)|}{\sqrt{N}},
 \end{align}
 where 
 \begin{equation}\label{eq:exp sum}
     S(N):=\sum_{n} 
 \, \lambda_{f}(n)\lambda_{g}(n) w_N(n). 
 \end{equation}

 \begin{lemma}[Voronoi Formula]\label{prop:voronoi}
Let $h\in C_{c}^{\infty}(0, \infty)$, $a, c\in\mathbb{Z}$ and $p$ be a prime such that $(ap,c)=1$. 
 If $f\in S_{k}^{*}(p, \chi)$, then  
\begin{align}
     \sum_{n\ge 1} \lambda_{f}(n)e\left(\frac{an}{c}\right)h(n) \ = \  2\pi i^{k} \frac{\chi(-c)}{c} \frac{\eta_{f}(p)}{\sqrt{p}} \sum_{n\ge 1} \ \overline{\lambda_{f}}(n) e\left(\frac{-\overline{ap}n}{c}\right) \int_{0}^{\infty} \ h(x) J_{k-1}\left(\frac{4\pi}{c}\sqrt{\frac{nx}{p}}\right) \ dx.
\end{align}
 \end{lemma}

\begin{proof}
   See \cite[Appendix A]{KMV02}. 
\end{proof}

Let $f \in S_k^{*}(p, \chi)$. We now state some properties for the Hecke eigenvalues $\lambda_f(n)$, which will be frequently used. We have the Hecke relation: 
\begin{equation}
\label{eq:HeckeRelation}
    \lambda_f(mn) = \sum_{d|(m,n)} \lambda_f(m/d) \lambda_f(n/d) \mu(d) \chi(d), 
\end{equation}
and by M\"{o}bius inversion, it follows that
$$
\label{eq:HeckeRelation0}
   \lambda_f(m) \lambda_f(n) = \sum_{d|(m,n)} \lambda_f(mn/d^2) \chi(d).
$$
Recall that 
\begin{equation}
\label{eq:HeckeRelation2}
    \lambda_f(n) = \chi(n) \overline{\lambda_f(n)}, \qquad \text{for} \qquad  (n,p) = 1;
\end{equation}
see \cite[p. 372]{IK04}, and
\begin{align}
\label{BoundAtp}
    |\lambda_f(p)|\leq 1; 
\end{align}
see \cite[Prop A.2]{KMV02}.  We also have the following $\ell_2$-bound from  Iwaniec \cite{IwaniecSpectralGrowth}:
\begin{align}\label{RS}
        \sum_{n\ll N}|\la_f(n)|^2\ll p^\epsilon N^{1+\varepsilon}.
\end{align}

\subsection{Archimedean analysis}
First we recall the definition of a family of inert functions.  Suppose $\mathcal{F}$ is an index set and $X: \mathcal{F} \rightarrow \mr_{\geq 1}$ is a function. We denote $X(i)$ as $X_i$, for $i \in \mathcal{F}$.
\begin{defn}
 A family of smooth functions $\{ w_i : i \in \mathcal{F} \}$ supported on a Cartesian product of dyadic intervals in $\mr_{>0}^d$ is called $X$-inert if for each $j = (j_1, \dots, j_d) \in \mz_{\geq 0}^d$ we have
 \begin{equation}
  \sup_{i \in \mathcal{F}}
  \sup_{(x_1, \dots, x_d) \in \mr_{>0}^d}
  \frac{x_1^{j_1} \dots x_d^{j_d} |w_i^{(j_1, \dots, j_d)}(x_1, \dots, x_d)|}{ X_i^{j_1 + \dots + j_d}} < \infty.
 \end{equation}
\end{defn}
For brevity, but as an abuse of terminology, we may refer to a single function being inert when we should properly say that it is part of an inert family.
In practice,  we will often have functions depending on a number of variables and it is unwieldy to list them all. 
Our convention will be to write $w(x_1, \dots, x_k; \cdot)$ where `$\cdot$' represents some suppressed list of variables for which $w$ is inert; we will use this notation even for $k=0$.

We now state some properties of inert families. Interested readers can find the proofs of these lemmas in \cite{BKY} and \cite[Sec. 2 and 3]{KPY}.

\begin{lemma}[Integration by parts]
\label{lemma:integrationbyparts}
Suppose that $\{ w_i(t) \}$ is a family of $X$-inert functions supported on $[Z, 2Z]$, with $w^{(j)}(t) \ll (Z/X)^{-j}$, for each $j=0,1,2, \dots$.  Suppose that $\phi$ is smooth and satisfies $\phi^{(j)}(t) \ll \frac{Y}{Z^j}$ for $j \geq 1$, for some $Y/X \geq R \geq 1$ and all $t$ in the support of $w$.  If $|\phi'(t)| \gg \frac{Y}{Z}$ for all $t \in [Z,2Z]$ then for arbitrarily large $A > 0$ we have
\begin{equation}
 \intR w(t) e^{i \phi(t)} dt \ll_A Z R^{-A}.
\end{equation}
\end{lemma}

\begin{prop}[Stationary phase] 
\label{prop:statphase}
 Suppose $\{ w_i(t_1, \dots, t_d) \}$ is an $X$-inert family supported on $t_1 \asymp Z$ and $t_k \asymp X_k$ for $k=2,\dots, d$.  Suppose that on the support of $w_i$, that $\phi = \phi_i$ satisfies
 \begin{equation}
  \frac{\partial^{a_1 + \dots + a_d}}{\partial t_1^{a_1} \dots \partial t_d^{a_d}}
  \phi(t_1, \dots, t_d) \ll \frac{Y}{Z^{a_1}} \frac{1}{X_2^{a_2} \dots X_d^{a_d}},
 \end{equation}
for all $(a_1, \dots, a_d) \in \mathbb{Z}_{\geq 0}$.  Suppose $\phi''(t_1, t_2, \dots, t_d) \gg \frac{Y}{Z^2}$ (here and below $\phi'$ and $\phi''$ refer to the derivative with respect to $t_1$) on the support of $w$.  Moreover, suppose there exists $t_0 \in \mr$ such that $\phi'(t_0) = 0$, and that $Y/X^2 \geq R \geq 1$.  Then
\begin{equation}
 \intR e^{i \phi(t_1, \dots, t_d)} w(t_1, \dots, t_d) dt_1
 = \frac{Z}{\sqrt{Y}} e^{i \phi(t_0, t_2, \dots, t_d)} W(t_2, \dots, t_d) + O(Z R^{-A}),
\end{equation}
where $W = W_i$ is some new family of $X$-inert functions.

\end{prop}

\begin{prop}\label{JBesAs}
    For $x \gg 1$, we have
    \begin{align}
        J_{k-1}\left( 2\pi x\right) \ = \ W_{k}(x) e(x)  \ + \ \overline{W_{k}}(x) e(-x)
    \end{align}
    for some function $W_k$ satisfying $x^{j}(\partial^{j}W_{k})(x) \ll_{j,k} 1/\sqrt{x}$. 
\end{prop}

\subsection{Elementary bounds}
We will repeatedly use the following elementary estimates.  First, we have $|S(0,n;c)| \leq (n,c)$, and secondly, we have for $n \geq 1$ that
\begin{equation}
\label{eq:gcdsum}
    \sum_{c \leq X} (n,c) \ll X (nX)^{\varepsilon},
    \qquad \text{and so}
    \qquad
    \sum_{c \leq X} |S(0,n;c)| \ll X (nX)^{\varepsilon}.
\end{equation}

\section{The proof}\label{sect:proof}

We now present the proof of Theorem \ref{thm:mainthm}. In this section, we will focus on $f$ and $g$ being holomorphic cusp forms, for simplicity of notation. We will present the adjustments needed when $f$ is a Maass form and/or $g$ is either a Maass form or Eisenstein series in Section \ref{sect:adjustment}.  
To aid in this generalization, we will 
not use the Deligne bound and instead use only \eqref{RS}.

\subsection{Amplification}\label{sect.Amp}

Let $2\leq L\leq p^{1-\varepsilon}$ be a parameter and  $\gamma_{\nu}$ be the indicator function on primes with $L/2 < \nu \leq L$.  Define
\begin{align}
\mathcal A(L):=\mathcal A_1(L)+\mathcal A_2(L), 
\end{align}
where \begin{align} 
    \mathcal{A}_j(L):= \sum_{\nu \leq L}\gamma_{\nu} a_j(\nu)\lambda_f(\nu^j),   \qquad \text{where} \qquad a_j(\nu)=\begin{cases}
         \overline{\lambda_f(\nu)} & j=1,\\
        - \overline{\chi(\nu)} & j=2.
    \end{cases}
\end{align}
For $j=1, 2$, define
\begin{equation}
\label{eq:Ljdef}
    \mathcal L_j= \{ \nu^j : \nu \text{ prime}, \thinspace L/2 < \nu \leq L\},
\end{equation}
and alternatively write 
\begin{equation}
\mathcal{A}_j(L) = \sum_{\ell \in \mathcal{L}_j} b_j(\ell) \lambda_f(\ell),
\qquad \text{where} \qquad 
b_j(\ell) =
\begin{cases}
    \overline{\lambda_f(\ell) }, &j=1, \\
    \overline{\chi}(\sqrt{\ell}), &j =2.
\end{cases}
\end{equation}
Note that $|\mathcal{L}_1| = |\mathcal{L}_2| = L^{1+o(1)}$, and observe from \eqref{RS} and the definition of $b_j(\ell)$ that
\begin{equation}
\label{bjl2bound}
    \sum_{\ell \in \mathcal{L}_j} |b_j(\ell)|^2 \ll L p^{\varepsilon}.
\end{equation}

The amplifier $\mathcal A(L)$ will serve as an unconditional replacement for the conjectural estimate in \eqref{eq:conditionalamplifier}.  Indeed, by
\eqref{eq:HeckeRelation2} and the prime number theorem we have 
\begin{align}\label{eq:lower bound}
    \mathcal{A}(L)=&\ \mathcal{A}_1(L) + \mathcal{A}_2(L)= \sum_{\nu \leq L} \gamma_\nu (|\lambda_f(\nu)|^2 -\overline{\chi(\nu)}\lambda_f(\nu^2)) \nonumber\\
    =&\ \sum_{\nu \leq L} \gamma_\nu \overline{\chi(\nu)}(\lambda_f(\nu)^2 -\lambda_f(\nu^2))= \sum_{\nu \leq L} \gamma_\nu\gg L(pL)^{-\varepsilon}.
\end{align} 
For  a positive integer $r \geq 1$, define
\begin{equation}
S(N,r) := \sum_n \lambda_f(nr) \lambda_g(n) w_N(n).
\end{equation}
Note that $S(N,1)=S(N)$ as defined in  \eqref{eq:exp sum}. 

Now consider
an amplified sum of the form 
\begin{align}\label{SnAdef}
S(N, \mathcal{A}) = S(N, \mathcal{A}_1) + S(N, \mathcal{A}_2),
\end{align} where 
\begin{equation}
S(N, \mathcal{A}_j)
:=
\sum_{\nu } \gamma_{\nu } a_j(\nu)
S(N, \nu^j)=\sum_{\ell \in \mathcal L_j}b_j(\ell)S(N, \ell) \hspace{20pt} j=1,2. 
\end{equation}

\begin{lemma}
\label{lemma:amplifiedSum}
We have
     \begin{equation}
        S(N) = \frac{S(N, \mathcal{A})}{\mathcal{A}(L)} + O\Big( (pLN)^{\varepsilon} \frac{N}{L}  \Big).
    \end{equation}
\end{lemma}

\begin{proof}
Using the Hecke relation \eqref{eq:HeckeRelation}, we see that  for $j=1,2$
\begin{equation}
    S(N, \mathcal A_j)
= \sum_{\nu } \gamma_{\nu} a_j(\nu)
\sum_{d| \nu^j}
\lambda_f(\nu^j/d) \mu(d) \chi(d)
\sum_{n \equiv 0 \shortmod{d}} 
\lambda_f(n/d) 
\lambda_g(n) w_N(n).
\end{equation}
The contribution from the sub-sum with $d=1$ is given by 
$S(N) \cdot \mathcal{A}_j(L)$.  The sub-sum with $d= \nu$ gives
\begin{equation}
    -\sum_{\nu} \gamma_{\nu} a_j(\nu) \lambda_f(\nu^{j-1})
    \chi(\nu)
    \sum_n \lambda_f(n) \lambda_g(n \nu) w_{N}(\nu n).
\end{equation}
Using \eqref{RS}, we see that for $j=1,2$ the previous expression is bounded by:
\begin{equation}
    \sum_{\substack{n, \nu\\ \nu n\asymp N}} |\lambda_f(\nu) \lambda_f(n)| \cdot | \lambda_g(n \nu)|
    \ll \sum_{\substack{n, \nu\\ \nu n\asymp N}} \left(|\lambda_f(\nu) \lambda_f(n)|^2 + | \lambda_g(n \nu)|^2\right) \ll N p^{\varepsilon}.
\end{equation}
Hence
\[ S(N, \mathcal A)= S(N) \mathcal A(L)+O((pLN)^{\varepsilon} N). \]
The lemma follows from the lower bound \eqref{eq:lower bound}.
\end{proof}

\subsection{Proof of Theorem \ref{thm:mainthm} assuming the key proposition}
The main technical result of this paper is the following.
\begin{prop}
\label{thm:SNAbound}
    Let $\sigma = 1/20$. Suppose that $L \ll N^{1/10}$ and $N \ll p^{1+\varepsilon}$.  Then for $j=1,2$
   \begin{equation}
   \label{eq:SNboundwithj}
       S(N, \mathcal{A}_j) \ll p^{\varepsilon}\Big( \sqrt{N L p} +N^{1-7\sigma/4}p^{3\sigma/2}L^{3j/2+1-9j\sigma/4}\Big).
   \end{equation} 
\end{prop}
Now we deduce  Theorem \ref{thm:mainthm} from Proposition \ref{thm:SNAbound} and Lemma \ref{lemma:amplifiedSum}. 
Indeed, together they imply
\begin{equation}
S(N) \ll  \frac{p^{\varepsilon}}{L} 
\Big( \sqrt{N L p} + N^{1-7\sigma/4}p^{3\sigma/2}L^{3+1-9\sigma/2}\Big).
\end{equation}
We choose $L=N^{-\frac{2 - 7\sigma}{14-18\sigma}} p^{\frac{1-3\sigma}{7-9\sigma}}$
when $N \geq p^{2/3}$ which ensures that $L \ll N^{1/10}$,
which gives
\begin{align}
S(N)\ll p^\varepsilon
N^{\frac12 + \frac{2 - 7 \sigma}{28 - 36 \sigma}} p^{\frac{3(1 - \sigma)}{7 - 9 \sigma}}.
\end{align}
Together with the trivial bound $S(N)\ll p^\varepsilon N \ll p^{1/3 + \varepsilon} N^{1/2}$ for $N \leq p^{2/3}$ (via \eqref{RS}),
it follows that 
  \begin{align}
  S(N)\ll p^{\varepsilon} \sqrt{N}(  N^{\frac{2-7
  		\sigma}{28-36\sigma}}p^{\frac{3(1-\sigma)}{7-9\sigma}+\varepsilon}+ p^{1/3}).
  \end{align}
Hence by \eqref{eq:AFEdyadicVersion} we deduce with $\sigma=1/20$ that
\begin{equation}\label{final bound}
    L(1/2, f \otimes g) \ll p^{1/2-1/524+\varepsilon}.
\end{equation}

If we assume $\mathcal{A}_1(L) \gg L p^{-\varepsilon}$, then we can get a better bound by taking $j=1$ in \eqref{eq:SNboundwithj}.
The choice of $L=N^{-\frac{2-7\sigma}{8-9\sigma}}p^{\frac{2-6\sigma}{8-9\sigma}}$ with $\sigma=1/20$ leads to
\begin{equation}
    L(1/2, f \otimes g) \ll p^{1/2-1/302+\varepsilon},
\end{equation}
as claimed in Remark \ref{remark:conditionalamplifier}.
The rest of the paper is devoted to proving Proposition \ref{thm:SNAbound}.

\subsection{Applying the delta symbol}

As a standard opening move, we write
\begin{equation}
    S(N,\ell)
    = 
    \sum_{m, n} \lambda_f(m) \lambda_g(n) \delta(m=n\ell) w_N(n) w_{N \ell}(m),
\end{equation}
where $w_{N \ell}$ is  another smooth and compactly-supported test function which is identically $1$ for $m=n\ell$ and $n$ in the support of $w_N$. 
Applying the delta symbol in Lemma \ref{lemma:DeltaSymbol} with 
\begin{equation}
\label{eq:Cdef}
C = \sqrt{NL^j} \asymp \sqrt{N \ell},
\end{equation}
directly gives
\begin{equation}
\label{afterdelta}
    S(N, \ell)
    = 
    \sum_{i=1}^{2} 
    \sum_{d} 
    \sum_{c} \frac{F_i(cd/C)}{cdC}
    \sumstar_{h \shortmod{c}}
\sum_m \lambda_f(m) w_{N \ell}(m)
e_c(hm) 
\sum_n \lambda_g(n) w_N(n)
e_c(-hn \ell)
      U_i\Big(\frac{m-n\ell}{cdC}\Big).
\end{equation}
Write $S(N, \ell) = \sum_{i=1}^{2} S_i(N, \ell)$.  Both terms will be estimated in the same way, and for simplicity of notation, we redefine $S(N, \ell)$ as either one of $S_i(N, \ell)$.
We further write 
\begin{equation}\label{SS0S'}
    S(N, \ell) = S_0(N, \ell) + S'(N, \ell),
\end{equation} where $S_0(N,\ell)$ denotes the contribution of the  terms with $(\ell, c) \neq 1$ and $S'(N,\ell)$ denotes the remaining contributions from the terms with  $(\ell, c) = 1$.

\subsection{Voronoi and stationary calculus}
\label{section:FunctionalEquations}
The next step is to apply the Voronoi formula to $S'(N, \ell)$.
\begin{lemma}
\label{lemma:S'NellPostVoronoi}
For  $C<p^{1-\varepsilon}$,  we have
 \begin{equation}
        S'(N, \ell) = \frac{1}{C \sqrt{p}}
        \sum_{\substack{M', N', Q \\ \text{dyadic}}}
    \sum_{d} 
    \sum_{\substack{(c, \ell) = 1 \\ c \asymp Q}}  \frac{\chi(c) }{c^3 d}      \, 
    \sum_{m \asymp M'} \, 
    \sum_{n \asymp N'} \, \overline{\lambda_{f}}(m)\lambda_{g}(n) 
           S(0, pn- m \ell; c)
           I(m,n,\cdot),
    \end{equation}   
    where  for some inert weight function $w(\cdot)$ we have
    \begin{equation}
    \label{eq:Imndef}
        I(m,n, \cdot) = 
        \int_{0}^{\infty} \, \int_{0}^{\infty} \ w_{N\ell}(x)w_{N}(y) U\left( \frac{x-y\ell}{cdC}\right) J_{k-1}\left(\frac{4\pi}{c}\sqrt{\frac{mx}{p}}\right)J_{k'-1}\left(\frac{4\pi\sqrt{ny}}{c}\right) 
        w(\cdot)
        \ dx\, dy.
    \end{equation}
     \end{lemma}
\begin{proof}
Since $C<p^{1-\varepsilon}$, we have $(c,p)=1$.  By assumption, we also have $(c, \ell) = 1$.  Hence applying the Voronoi summation in Lemma \ref{prop:voronoi} to the $m$ and $n$-sums in $S'(N, \ell)$, we obtain
\begin{multline}
\label{afteVor2}
    S'(N, \ell) =  (2\pi i^{k}) (2\pi i^{k'})\frac{\eta_{f}(p)}{\sqrt{p}}   
    \sum_{d} 
    \sum_{(c, \ell) =1}  \frac{F(cd/C)}{cdC}    \frac{ \chi(-c)}{c^2}   \sum_{m}  \sum_{n}
    \overline{\lambda_{f}}(m)\lambda_{g}(n) S(0, \bar{\ell}n-\bar{p}m; c) 
    \\
      \times \int_{0}^{\infty} \int_{0}^{\infty} \ w_{N\ell}(x)w_{N}(y) U\left( \frac{x-y\ell}{cdC}\right) J_{k-1}\left(\frac{4\pi}{c}\sqrt{\frac{mx}{p}}\right)J_{k'-1}\left(\frac{4\pi\sqrt{ny}}{c}\right) \ dx\, dy,   
\end{multline}
where the $h$-sum has been evaluated as a Ramanujan sum. We apply  dyadic partitions of unity to the sums over $m$, $n$ and $c$: $m \asymp M'$, $n \asymp N'$ and $c \asymp Q$, say. Next we absorb $4\pi^2 i^{k+k'}$, $\eta_f(p)$, $\chi(-1)$, and $F(cd/C)$ into an inert weight function $w(\cdot)$, and define $I(m,n,\cdot)$ as in \eqref{eq:Imndef} to complete the proof. 
\end{proof}
In the following, we record the asymptotic behaviour of the integral $I(m,n,\cdot)$.
\begin{lemma}[Non-oscillatory range]
\label{lemma:ImnNonOscillatory}
    Suppose that
    \begin{equation}
    \label{eq:N'phase}
 \frac{\sqrt{N' N}}{Q} \ll p^{\varepsilon}.
    \end{equation}
    Then $I(m,n,\cdot)$ is very small unless
    \begin{equation}
        \label{eq:M'phase}
        \frac{\sqrt{M' N \ell}}{Q \sqrt{p}} \ll p^{\varepsilon}.
    \end{equation}
  Assuming \eqref{eq:N'phase} and \eqref{eq:M'phase} hold, and
  with $w(\cdot)$ being a $p^{\varepsilon}$-inert function, we have
  \begin{equation}
  \label{eq:ImnNON}
      I(m,n,\cdot) = QdC N w(\cdot).
  \end{equation}
\end{lemma}

\begin{proof}
Change  variables $y = \frac{x}{\ell} + t$ in \eqref{eq:Imndef}; the inner $t$-integral takes the form 
\begin{equation}
    \intR w_{N}\Big(\frac{x}{\ell} + t \Big) U\Big(\frac{-t \ell}{cdC}\Big)
    w(\cdot) dt,
\end{equation}
where $w(\cdot)$ is $p^{\varepsilon}$-inert (absorbing $J_{k'-1}$ into the inert function).  The integral evaluates to  $\frac{QdC}{\ell} w(\cdot)$.  The remaining $x$-integral then has a single phase of size $\frac{\sqrt{M' N \ell}}{Q \sqrt{p}}$. Hence, integration by parts shows the integral is very small unless \eqref{eq:M'phase} holds.  Assuming \eqref{eq:M'phase},  we can absorb $J_{k-1}$  into the inert weight function, and we see that the $x$-integral has size $N \ell$, resulting in the desired form of $I(m,n,\cdot)$.
\end{proof}

\begin{lemma}[Oscillatory range]
\label{lemma:ImnOscillatory}
 Let
\begin{equation}
\label{eq:Zdef}
    z := \frac{Q \sqrt{M' L^j p}}{\sqrt{N}}, \quad P:=\frac{\sqrt{M' N L^j}}{Q \sqrt{p}}.
\end{equation}
Suppose that
\begin{equation}
\label{eq:N'oscillatoryRange}
    \frac{\sqrt{N' N}}{Q} \gg p^{\varepsilon}.
\end{equation}
Then $I(m,n,\cdot)$ is very small unless 
\begin{equation}
\label{eq:M'N'samesizeOscillatory}
    P\asymp \frac{\sqrt{N' N}}{Q},
\end{equation}
and
\begin{align}
\label{truncation}
    M' \ll \frac{N L^j p^{1+\varepsilon}}{d^2C^2}=\frac{p^{1+\varepsilon}}{d^2} \quad \text{ and } \quad N' \ll \frac{NL^{2j}p^\varepsilon}{d^2C^2}=\frac{L^{j}p^\varepsilon}{d^2}.
\end{align}
In this case, we have 
\begin{equation}
\label{eq:ImnOSC}
    I(m,n, \cdot)
    = \frac{Q d CN}{P} w\Big(\frac{m \ell - np}{zp^\varepsilon}, \cdot \Big),
\end{equation}
where $w(x, \cdot)$ is $p^{\varepsilon}$-inert in the suppressed variables, has compact support on $x \ll 1$, and is ``essentially $p^{\varepsilon}$-inert" in terms of $x$, meaning
$\frac{d^j}{dx^j} w(x, \cdot) \ll (p^{\varepsilon})^j$.
\end{lemma}
\begin{proof}
First, we note that integration by parts shows $I(m,n,\cdot)$ is very small unless \eqref{truncation} holds.
Next,
changing variables $x = y \ell + u$ gives
\begin{align}
I(m,n,\cdot) = 
    \int_{0}^{\infty} \, \int_{-\infty}^\infty \ w_{N\ell}(y\ell+u)w_{N}(y) U\left( \frac{u}{cdC}\right) 
    J_{k-1}\left(\frac{4\pi}{c}\sqrt{\frac{m(y\ell+u)}{p}}\right)
    J_{k'-1}\left(\frac{4\pi\sqrt{ny}}{c}\right) \ du\, dy.  
\end{align}
By assumption \eqref{eq:N'oscillatoryRange}, $J_{k'-1}$ is oscillatory on the region of integration.  This then forces the other Bessel function $J_{k-1}$ to have some oscillation as well, as otherwise the entire integral is very small.  
This shows that $I(m,n,\cdot)$ is very small unless \eqref{eq:M'N'samesizeOscillatory} holds.

By Taylor expansions, we have
\begin{equation}
    \frac{\sqrt{m (y \ell + u) }}{c \sqrt{p}}
    = \frac{\sqrt{m y \ell }}{c \sqrt{p}} \Big(1 + \frac{u}{2y \ell} + O\Big(\frac{u^2}{y^2 \ell^2}\Big)\Big).
\end{equation}
Using \eqref{truncation} and $cdC \ll C^2 \asymp N \ell$, the size of the  quadratic term above is
\begin{equation}
\ll \frac{\sqrt{m y \ell }}{c \sqrt{p}}  \frac{u^2}{y^2 \ell^2}
\ll \frac{N \ell p^{\varepsilon}}{cdC } \frac{(cdC)^2}{N^2 \ell^2} = p^{\varepsilon} \frac{cdC}{N \ell} \ll p^{\varepsilon}.
\end{equation}
Hence, by Proposition \ref{JBesAs},
we have
\begin{equation}
    J_{k-1}\left(\frac{4\pi}{c}\sqrt{\frac{m(y\ell+u)}{p}}\right)
    =  P^{-1/2}
    e\left(\frac{2 \sqrt{m y \ell}}{c \sqrt{p}}+\frac{ \sqrt{m}  u}{   c \sqrt{y\ell p}} \right) w(\cdot)+ P^{-1/2}
    e\left(-\frac{2 \sqrt{m y \ell}}{c \sqrt{p}}-\frac{ \sqrt{m}  u}{   c \sqrt{y\ell p}} \right) \overline{w}(\cdot).
\end{equation}

In total, $I(m,n,\cdot)$ is a linear combination (over $\pm$ signs) of integrals of the form
\begin{equation}
P^{-1}
 \int_{0}^{\infty} \, \int_{-\infty}^\infty \ w_{N\ell}(y\ell+u)w_{N}(y) U\left( \frac{u}{cdC}\right) 
  e\left(\pm \frac{2 \sqrt{m y \ell}}{c \sqrt{p}}  \pm \frac{2 \sqrt{ny}}{c} \right) 
    e\left(\frac{\pm \sqrt{m}  u}{ c \sqrt{y\ell p}} \right)
    w(\cdot) du dy,
\end{equation}
where $w(\cdot)$ is $p^\varepsilon$-inert in all variables.
The $u$-integral simply evaluates as $QdC$ times a bump function, which enforces the effective support condition \eqref{truncation}.  Hence, the integral simplifies further as
\begin{equation}\label{SimplifiedIntegral}
 \frac{QdC N}{P}
 \int_{y \asymp 1}
  e\left(\pm \frac{2 \sqrt{m y N \ell}}{c \sqrt{p}}  \pm \frac{2 \sqrt{nN y}}{c} \right) 
    w(y,\cdot) dy.
\end{equation}
When the two $\pm$ signs are equal, then the $y$-integral is very small by integration by parts (recalling \eqref{eq:N'oscillatoryRange}).
When the two $\pm$ signs are opposites,  
after changing variables $y \rightarrow y^2/4$ and re-defining the inert function $w$, we get (for one of the two sign combinations)
\begin{equation}
    \int e\Big( y \Big(\frac{ \sqrt{m N \ell}}{c \sqrt{p}} - \frac{ \sqrt{nN}}{c}\Big) \Big) w(y,\cdot) dy = \widehat{w}\Big(\frac{ \sqrt{m N \ell}}{c \sqrt{p}} - \frac{ \sqrt{nN}}{c},\cdot \Big)
    = \widehat{w}\Big(\frac{\sqrt{m \ell} - \sqrt{n p}}{c \sqrt{p}/\sqrt{N}},\cdot \Big).
\end{equation}
This function may be replaced by a function of the form $w(\frac{m \ell - np}{zp^\varepsilon}, \cdot)$. This completes the proof.
    \end{proof}
    
For later use, we record some estimates that follow from Lemmas \ref{lemma:ImnNonOscillatory} and \ref{lemma:ImnOscillatory}:
\begin{equation}
\label{eq:M'N'over1plusPbound}
    \frac{M' N'}{1+P} \ll 
    \frac{Q p^{1+\varepsilon}  L^{j/2}}{d^3 \sqrt{N}},  
    \qquad M' \ll \frac{p^{1+\varepsilon}}{d^2},
    \qquad N' \ll \frac{L^j p^{\varepsilon}}{d^2},
\end{equation}
which hold uniformly in both the oscillatory and non-oscillatory ranges.

It is time to place the sum over $\ell$ back into the definition.  Let
\begin{equation}
\label{eq:SO'NA}
    S'(N, \mathcal{A}_j)
    = \sum_{\nu} \gamma_{\nu} a_j(\nu)S'(N,\nu^j)=
    \sum_{\ell\in\mathcal{L}_j} b_j(\ell) 
    S'(N, \ell).
\end{equation} The condition $C<p^{1-\varepsilon}$ is satisfied as long as $L<p^{1/2-\varepsilon}$. Combining Lemma \ref{lemma:ImnNonOscillatory} and Lemma \ref{lemma:ImnOscillatory}, together with the fact that $\frac{z}{P}=\frac{Q^2p}{N}$, we conclude that: 
\begin{lemma}\label{lemma:combined}
Let \begin{align}\label{Zdef}
    Z:=z(1+P^{-1})p^\varepsilon=
    p^\varepsilon\frac{Q \sqrt{M' L^j p}}{\sqrt{N}}\Big(1+\frac{Q \sqrt{p}}{\sqrt{M' N L^j}}\Big)
    = p^{\varepsilon} \Big(\frac{Q \sqrt{M' L^j p}}{\sqrt{N}} + \frac{Q^2 p}{N}\Big).
\end{align}
For $j=1, 2$, we have
\begin{align}
 S'(N, \mathcal{A}_j)\ll & \  p^\varepsilon\sum_d \sum_{\substack{M',N' : \,  \eqref{truncation}\\ Q\ll C/d \\ \mathrm{dyadic}}}\frac{N}{(1+P) Q^2  \sqrt{p}}  \sum_{m \asymp M'} |{\lambda_{f}}(m)|\nonumber\\
&\ \times 
    \Big| \sum_{\substack{(c, \ell) = 1 \\ c \asymp Q}}  \chi(c)        \sum_{\ell\in\mathcal{L}_j} b_j(\ell)
    \sum_{n \asymp N'} \, \lambda_{g}(n) 
           S(0, pn- m \ell; c)
 W\Big(\frac{m \ell - np}{Z }, \cdot\Big)\Big|.
\end{align}
where $W(x,\cdot)$ is a $p^\varepsilon$-inert function supported on $[-1,1]$ for the first variable and on dyadic intervals for the rest of the variables.  
\end{lemma}
For later convenience, we note 
the following unified bound:
\begin{equation}
\label{eq:ZboundUnified}
    Z \ll \frac{Q p^{1+\varepsilon} L^{j/2}}{d \sqrt{N}}.
\end{equation}

\subsection{Bounding the contribution of \texorpdfstring{$S_0(N, \ell)$}{S0(N,l)}}\label{sect:BoundingS0}

Similar to the notation $S'(N,\mathcal{A}_j)$, write 
\begin{align}
    S_0(N, \mathcal{A}_j)=\sum_\nu \gamma_{\nu} a_j(\nu)S_0(N,\nu^j)= \sum_{\ell\in\mathcal{L}_j} b_j(\ell) 
    S_0(N, \ell).
\end{align}
\begin{lemma}
\label{lemma:S0NAbound}
We have 
    \begin{equation}
        S_0(N, \mathcal{A}_1)\ll p^{-2024} \quad \text{ and } \quad S_0(N, \mathcal{A}_2) \ll p^{\varepsilon} \sqrt{Np}.
    \end{equation}
\end{lemma}
Note that Lemma \ref{lemma:S0NAbound} is more than satisfactory for Proposition \ref{thm:SNAbound}, so 
an improvement here would not affect the main theorem.
\begin{proof}
Recall 
$\mathcal{L}_j$ was defined in \eqref{eq:Ljdef}.
The condition $(\ell, c) > 1$ means $\nu \mid c$.  We write
\begin{equation}
    e\Big(\frac{-hn \ell}{c}\Big) = e\Big(\frac{- hn\ell' }{c'}\Big), \hspace{15pt} \text{where} \hspace{15pt} \ell' := \frac{\ell}{(c, \ell)}\hspace{10pt} \text{and} \hspace{10pt} c' := \frac{c}{(c, \ell)},
\end{equation}
and apply the Voronoi formula with this modified exponential. 
We can follow the argument of Lemma \ref{lemma:S'NellPostVoronoi}. The standard calculation of the dependence of the dual sum on the conductor and the length of the original sum shows that the $n$-sum may be truncated earlier (compared to the analysis in Section \ref{section:FunctionalEquations}) by a factor of $(c/c')^2$. Thus, we have
\begin{align}
        S_0(N, \mathcal{A}_j) =&  \frac{1}{C \sqrt{p}}
       \sum_{d} \sum_{\substack{M', N':  
        \eqref{truncation}\\ Q\ll C/d \\ \text{ dyadic} 
        }}
        \sum_{L/2<\nu\leq L} \gamma_{\nu} a_j(\nu)
    \\
    & \times 
    \sum_{\substack{c \equiv 0 \shortmod{\nu} \\ c \asymp Q}}  \frac{\chi(c) }{c^2 c' d}      \, 
    \sum_{m \asymp M'} \, 
    \sum_{n \asymp N'\left(\frac{c'}{c}\right)^2} \, \overline{\lambda_{f}}(m)\lambda_{g}(n) 
           S(0, \overline{\ell'} (c,\ell) n- \overline{p} m; c)
           I_0(m,n,\cdot),
    \end{align}  
where
\begin{equation}
        I_0(m,n, \cdot) = 
        \int_{0}^{\infty} \, \int_{0}^{\infty} \ w_{N\ell}(x)w_{N}(y) U\left( \frac{x-y\ell}{cdC}\right) J_{k-1}\left(\frac{4\pi}{c}\sqrt{\frac{mx}{p}}\right)J_{k'-1}\left(\frac{4\pi\sqrt{ny}}{c'}\right) 
        w(\cdot)
        \ dx\, dy.
    \end{equation}
 If $j=1$ or $c/c' = \nu^2$ then recalling \eqref{truncation},
 this means the $n$-sum is essentially empty .  Therefore, we only need to consider the case $j=2$ with $\nu || c$, so $(c', \nu) = 1$.  
 
We focus on the oscillatory case since it is a little harder.
A modification of the ideas in Lemma \ref{lemma:ImnOscillatory} shows that $I_{0}(m,n, \cdot)$ is a linear combination of expressions of the form
\begin{equation}
    \frac{Q dCN}{P'} \int_{y \asymp 1} e\Big(\pm \frac{2 \sqrt{mNy \ell}}{c \sqrt{p}} \pm \frac{2 \sqrt{nNy}}{c'}\Big) w(y, \cdot) dy,
\end{equation}
where $P'$ is the size of the phase here. Note that $P'=P$ as in \eqref{eq:Zdef}.
With $M'$ and $N'$ satisfying \eqref{eq:M'N'over1plusPbound}, we have the truncations
\begin{equation}
    m \ll M', \qquad n \ll \frac{N'}{\ell} \asymp \frac{N}{L^2}.
\end{equation}

Using that $\ell' = \nu$, $(c, \ell) = \nu$, $(c', \nu) = 1$, etc., the Ramanujan sum factors as
\begin{equation}
    S(0, \overline{\ell'} (c,\ell) n- \overline{p} m; c' \nu)
    =S(0, p n-  m; c' ) S(0,  m; \nu).
\end{equation}
 Hence the contribution to $S_0(N, \mathcal{A}_2)$ from these terms 
is bounded by
\begin{equation}
\label{eq:NotFeelingCreative}
\frac{p^\varepsilon}{C\sqrt{p}}\sum_{\nu\asymp L}\sum_{d}   \sum_{\substack{M', N': 
        \eqref{truncation}\\ Q\ll C/d \\ \text{ dyadic} 
        }}
\sum_{\substack{c \equiv 0 \shortmod{\nu} \\ c \asymp Q}}
        \frac{1}{c^2c'd} 
 \sum_{\substack{m \ll M' \\ n \ll L^{-2} N'}}    
        \frac{QCN}{P}|\lambda_f(m)\lambda_g(n)|.
\end{equation}
Using $|\lambda_f(m)\lambda_g(n)|\ll |\lambda_f(m)|^2+|\lambda_g(n)|^2$
and \eqref{RS},
we obtain that
\eqref{eq:NotFeelingCreative} is
        \begin{equation}
 \ll \frac{p^\varepsilon}{C\sqrt{p}} L \sum_{d}   \sup_{\substack{M', N': 
        \eqref{truncation}\\ Q \ll C/d
        }} \frac{Q}{L} \frac{L}{Q^3} \frac{N'}{L^2}M'\frac{QCN}{P} 
        \ll p^\varepsilon\sqrt{Np},
\end{equation}
where we used \eqref{eq:M'N'over1plusPbound} to aid in simplification.
The easier non-oscillatory range is treated in a similar fashion.
\end{proof}

 Let $j=1,2$. Combining \eqref{SS0S'} and Lemma \ref{lemma:S0NAbound}, we have \begin{align}
    S(N,\mathcal{A}_j)\ll |S'(N,\mathcal{A}_j)|+p^\varepsilon \sqrt{Np} \cdot \delta(j=2)+p^{-2024}.
\end{align}
We also write
\begin{equation}
S '(N, \mathcal{A}_j) = 
\underbrace{S^{(0)}(N, \mathcal{A}_j)}_{p|m} + 
\underbrace{S''(N, \mathcal{A}_j)}_{p \nmid m}.
\end{equation}
\begin{lemma}
\label{lemma:trivialbound}
For $j=1,2$, we have
\begin{equation}\label{S(0)Bound}
    S^{(0)}(N, \mathcal{A}_j)
     \ll p^{-1/2+\varepsilon} N L^{j+1}.
\end{equation}    
\end{lemma}
\begin{proof}
From Lemma \ref{lemma:combined}, and the trivial bounds $(pn-m\ell,c)\leq c$ and $|w|\ll 1$, we have
  \begin{align*}
     S^{(0)}(N, \mathcal{A}_j)
     & \ \ll 
      \sum_d
        \sum_{\substack{M', N': \,  
        \eqref{truncation}\\ Q\ll C/d \\ \text{ dyadic} 
        }}
\frac{N p^{\varepsilon}}{(1+P) Q^2  \sqrt{p}} \sum_{\substack{p \mid m \\ m \asymp M'}} |\lambda_f(m)|
    \sum_{\substack{(c, \ell) = 1 \\ c \asymp Q}} c \sum_{\ell \in \mathcal{L}_j} |b_j(\ell)|     
    \sum_{n \asymp N'}   | \lambda_g(n)| .
    \end{align*}
    Using \eqref{eq:HeckeRelation} under the condition for $M'$ in \eqref{eq:M'N'over1plusPbound} and \eqref{BoundAtp}, we have
     $|\lambda_f(pm')| \leq |\lambda(m')| $, which together with \eqref{RS}, \eqref{bjl2bound} and \eqref{eq:M'N'over1plusPbound} (for simplification) gives
    \begin{align*}
S^{(0)}(N, \mathcal{A}_j) & \ \ll  \sum_{d}  \sum_{\substack{M', N': 
        \eqref{truncation}\\ Q\ll C/d \\ \text{ dyadic} 
        }}\frac{N p^\varepsilon}{(1+P)Q^2 \sqrt{p}} \frac{M' }{p}  Q^2LN'
        \ll p^\varepsilon \frac{N}{\sqrt{p}} L^{j+1}. \qedhere
\end{align*}
\end{proof}

\subsection{Cauchy--Schwarz and inclusion-exclusion}
It remains to estimate $S''(N, \mathcal{A}_j)$, which satisfies
\begin{align}
    |S''(N, \mathcal A_j)| \ll & \ p^\varepsilon\sum_d \sum_{\substack{M',N' \eqref{truncation}\\ Q\ll C/d \\ \text{dyadic}}}\frac{N}{(1+P) Q^2  \sqrt{p}} \sum_{\substack{p\nmid m \\ m \asymp M'}} |\lambda_f(m)|
    \\&\times
     \Big|\sum_{\substack{(c, \ell) = 1 \\ c \asymp Q}}  \chi(c) 
    \sum_{\ell\in\mathcal{L}_j}b_j(\ell)
    \sum_{n \asymp N'} \, \lambda_{g}(n) 
           S(0, pn- m \ell; c)
 W\Big(\frac{m \ell - np}{Z }, \cdot\Big)\Big|.
\end{align}
Using the Cauchy--Schwarz inequality and  \eqref{RS} we obtain that
\begin{equation}
\label{eq:SO'NamplifiedBoundviaTO'Namplified}
   |S''(N, \mathcal{A}_j)|
\ll 
\sum_d 
\sum_{\substack{M', N': \eqref{truncation}\\  Q\ll C/d\\  \text{ dyadic}
        }}
      \frac{\sqrt{M'} N p^{\varepsilon}}{(1+P)Q^2 \sqrt{p}}    
|T '(N, \mathcal{A}_j)|^{1/2},
\end{equation}
where
\begin{equation}
T'(N, \mathcal{A}_j)
= 
   \sum_{\substack{m \asymp M' \\ p \nmid m }} 
   \Big|
   \sum_{\ell \in \mathcal{L}_j}b_j(\ell)
    \sum_{\substack{(c, \ell) = 1 \\ c \asymp Q}}  \chi(c)      \, 
    \sum_{n \asymp N'} \, \lambda_{g}(n) 
           S(0, pn- m \ell; c)
 W\Big(\frac{m \ell - np}{Z}, \cdot \Big)
 \Big|^2.
\end{equation}
Using \eqref{eq:M'phase}, \eqref{eq:Zdef}, and \eqref{eq:M'N'samesizeOscillatory}, we see that
$\frac{\sqrt{M'} N}{(1+P) Q^2 \sqrt{p}} \ll \frac{p^{\varepsilon} \sqrt{N}}{Q \sqrt{L^j}}$, whence 
\begin{align}
\label{eq:SO'NamplifiedBoundviaTO'NamplifiedSimplified}
     |S''(N, \mathcal{A}_j)|
\ll p^\varepsilon
\sum_d 
\sum_{\substack{M', N': \eqref{truncation}\\  Q\ll C/d\\  \text{ dyadic}
        }}
      \frac{ \sqrt{N}}{Q \sqrt{L^j}}    
|T '(N, \mathcal{A}_j)|^{1/2}.
\end{align}
Expanding the square in $T'(N, \mathcal A_j)$ gives 
\begin{equation}
\label{eq:Tosc'def}
  T'(N, \mathcal{A}_j)
= 
\sum_{\ell_1, \ell_2 \in \mathcal{L}_j}b_j(\ell_1)\overline{b_j(\ell_2)}
\sum_{\substack{(c_1, \ell_1) = 1 \\ (c_2, \ell_2) = 1 \\ c_1 \asymp c_2 \asymp Q}}
\chi(c_1) \overline{\chi}(c_2)
\sum_{n_1, n_2 \asymp N'}
\lambda_{g}(n_1) \lambda_{g}(n_2) \cdot V',
\end{equation}
where $V' = 
    V'(\ell_1, \ell_2, n_1, n_2, c_1, c_2)$ is defined by
\begin{equation}
\label{eq:V'def}
V'=        \sum_{\substack{m \asymp M' \\ p \nmid m}}
S(0, p n_1 - m \ell_1;c_1) S(0, pn_2 - m \ell_2;c_2)
 W\Big(\frac{m \ell_1 - n_1 p}{Z}, \cdot\Big)
  W\Big(\frac{m \ell_2 - n_2 p}{Z}, \cdot\Big).
\end{equation}
Next for $j=1,2$ we write 
\begin{equation}
    T'(N, \mathcal{A}_j) 
    =
    \underbrace{T^{(0)}(N, \mathcal{A}_j)}_{n_1 \ell_2 = n_2 \ell_1}
    + \underbrace{T''(N, \mathcal{A}_j)}_{n_1 \ell_2 \neq n_1 \ell_1}.
\end{equation}

\begin{lemma}
\label{lemma:Tosc0bound}
    For $j=1,2$, and $L \ll p^{1/2-\varepsilon}$, we have
    \begin{equation}
    \label{eq:T0boundinLemma}
        T^{(0)}(N, \mathcal{A}_j) \ll p^{\varepsilon} 
        \frac{Q^2  L^{1+j} p}{d^4}
    \end{equation}
  and its contribution to $S''(N, \mathcal{A}_j)$ via \eqref{eq:SO'NamplifiedBoundviaTO'NamplifiedSimplified} is
\begin{equation}
\label{eq:T0termboundTowardsS}
    \ll p^{\varepsilon} \sqrt{N L p}.
\end{equation}
\end{lemma}
\begin{proof}
    Recall $p \nmid m$ by assumption, and note $(p, \ell) =1$ since $\ell < p$.  Thus $pn_i \neq m \ell_i$, so the arguments of the Ramanujan sums in \eqref{eq:V'def} are both non-zero.  Hence we may use 
    \eqref{eq:gcdsum} to estimate the $c_i$-sums.
Therefore, using $|ab|\ll |a|^2+|b|^2$, \eqref{RS} and \eqref{bjl2bound}, we have
\begin{align}
    T^{(0)}(N, \mathcal{A}_j) 
    \ll p^{\varepsilon}
\sum_{\substack{\ell_1, \ell_2 \in \mathcal{L}_j \\ n_1, n_2 \asymp N' \nonumber\\ \ell_1 n_2 = \ell_2 n_1}}
\left(|b_j(\ell_1)\lambda_g(n_2)|^2+|b_j(\ell_2)\lambda_g(n_1)|^2\right)Q^2
M'
\ll p^{\varepsilon} L N' Q^2 M',
\end{align}
which can be simplified to  \eqref{eq:T0boundinLemma} using \eqref{eq:M'N'over1plusPbound}. 
Inserting \eqref{eq:T0boundinLemma} into \eqref{eq:SO'NamplifiedBoundviaTO'NamplifiedSimplified} shows
the contribution to $S''(N, \mathcal{A}_j)$ is bounded by
\begin{equation*}
 p^{\varepsilon} \sum_{d}
 \sup_{Q\ll C/d}
 \frac{\sqrt{N}}{Q \sqrt{L^j}} \Big(\frac{Q^2 p L^{1+j}}{d^4} \Big)^{1/2}
    \ll p^{\varepsilon}\sqrt{N L p}. \qedhere
\end{equation*}
\end{proof}

Next we write by inclusion-exclusion
\begin{equation}
    \underbrace{T''(N, \mathcal{A}_j)}_{p \nmid m} 
    =
    \underbrace{T(N, \mathcal{A}_j)}_{\text{all $m$}}
    -
    \underbrace{T^{(00)}(N, \mathcal{A}_j)}_{p \mid m}.
\end{equation}
\begin{lemma}
\label{lemma:Tosc00bound}
    For $j=1, 2$, we have
    \begin{equation}
        T^{(00)}(N, \mathcal{A}_j) \ll 
        p^{\varepsilon} 
        \frac{L^{2+2j} Q^3}{ d^{4}}, 
    \end{equation}
    and its contribution to $S''(N, \mathcal{A}_j)$ via \eqref{eq:SO'NamplifiedBoundviaTO'NamplifiedSimplified}  is bounded by
    \begin{equation}
    \label{eq:STosc00bound}
        \ll p^{\varepsilon} N^{3/4} L^{1+j}.
    \end{equation}
\end{lemma}
\begin{proof}
    Using $|ab|\ll |a|^2+|b|^2$, we obtain
 \begin{align}
        T^{(00)}(N, \mathcal{A}_j) \ll & \
         p^{\varepsilon}
\sum_{\substack{\ell_1, \ell_2 \in \mathcal{L}_j \\ n_1, n_2 \asymp N' \\ \ell_1 n_2 \neq \ell_2 n_1}}\left(|b_j(\ell_1)\lambda_g(n_2)|^2+|b_j(\ell_2)\lambda_g(n_1)|^2\right)\nonumber\\
&\times \ 
\sum_{\substack{(c_1, \ell_1) = 1 \\ (c_2, \ell_2) = 1 \\ c_1 \asymp c_2 \asymp Q}}
\sum_{\substack{ m \asymp M' \\ p \mid m}}
(pn_1 - m \ell_1, c_1) \cdot 
(pn_2 - m \ell_2, c_2)
\Big|W\Big(\frac{m \ell_1 - n_1 p}{Z}\Big)
W\Big(\frac{m \ell_2 - n_2 p}{Z}\Big)\Big|.
    \end{align}    
With \eqref{RS} and \eqref{bjl2bound}, the contribution from the terms where simultaneously $pn_1 \neq m \ell_1$ and $p n_2 \neq m \ell_2$ contribute to $ T^{(00)}(N, \mathcal{A}_j)$ can be bounded by 
\begin{equation}
    \ll p^{\varepsilon} L^{2} Q^2 (N')^2\frac{M'}{p}\ll p^\varepsilon \frac{ Q^2 L^{2+2j}}{d^6},
\end{equation}
which contributes to $S''(N, \mathcal{A}_j)$ at most
\begin{equation}
p^{\varepsilon}
    \sum_{d}\sup_{Q\ll C/d}\frac{\sqrt{N}}{Q \sqrt{L^j}} 
    \left(\frac{ Q^2 L^{2+2j}}{d^6} \right)^{1/2}
    \ll p^{\varepsilon} \sqrt{N} L^{1+\frac{j}{2}}.
\end{equation}
This is better than claimed in \eqref{eq:STosc00bound}.

Next, consider the contribution from $p n_1 = m \ell_1$.  
From the assumption $\ell_1 n_2 \neq \ell_2 n_1$, this implies $p n_2 \neq m \ell_2$.
For the $c_2$-sum, we may then use \eqref{eq:gcdsum}.  Using a divisor function bound to control the sum over $m$, 
we see that these terms contribute to $T^{(00)}(N, \mathcal{A}_j)$ at most
\begin{equation}
   p^{\varepsilon} L^2 (N')^2 Q^3 \ll p^{\varepsilon} L^{2+2j} Q^3 d^{-4}.
\end{equation}
Thus the contribution to $S''(N, \mathcal{A})$ is then bounded by
\begin{equation}
  p^{\varepsilon}  \sum_d \sup_{Q\ll C/d}\frac{\sqrt{N}}{Q \sqrt{L^j}} 
  \Big(\frac{L^{2+2j} Q^3}{d^{4}}\Big)^{1/2}
    \ll p^{\varepsilon} N^{3/4} L^{1+\frac{3j}{4}},
\end{equation}
which is slightly better than claimed.
The same bound applies to the terms with $pn_2 = m \ell_2$.
\end{proof}

Collecting Lemmas \ref{lemma:S0NAbound}, \ref{lemma:trivialbound}, \ref{lemma:Tosc0bound} and \ref{lemma:Tosc00bound}, 
and using \eqref{eq:SO'NamplifiedBoundviaTO'NamplifiedSimplified},
we obtain the following lemma which concludes the analysis thus far.
\begin{lemma}\label{lemma:conclusionBeforePoisson}
   Suppose $L\ll p^{1/2-\varepsilon}$.  For $j=1,2$, we have
    \begin{align}
    \label{eq:SboundInTermsofT}
        S(N,\mathcal{A}_j)\ll \sum_d  
        \sum_{\substack{M', N': \eqref{truncation}\\  Q\ll C/d\\  \mathrm{dyadic}
        }}
        \frac{ \sqrt{N}}{Q \sqrt{L^j}}       
        |T(N, \mathcal{A}_j)|^{1/2}
        +p^\varepsilon
        \left(\sqrt{NLp}
        +\frac{NL^{1+j}}{\sqrt{p}}
        +N^{3/4} L^{1+j}
        \right),
    \end{align}
    where 
    \begin{equation}
    \label{eq:TNAdefLate}
        T(N, \mathcal{A}_j)
= 
\sum_{\substack{\ell_1, \ell_2 \in \mathcal{L}_j \\ n_1, n_2 \asymp N' \\ \ell_1 n_2 \neq \ell_2 n_1}}b_j(\ell_1)\overline{b_j(\ell_2)}
\sum_{\substack{(c_1, \ell_1) = 1 \\ (c_2, \ell_2) = 1 \\ c_1 \asymp c_2 \asymp Q}}
\chi(c_1) \overline{\chi}(c_2)
\lambda_{g}(n_1) \lambda_{g}(n_2) \cdot V,
\end{equation}
and $V = 
    V(\ell_1, \ell_2, n_1, n_2, c_1, c_2)  $ is defined by
\begin{equation}
\label{eq:Vdef}
V=        \sum_{\substack{m \asymp M'}}
S(0, p n_1 - m \ell_1;c_1) S(0, pn_2 - m \ell_2;c_2)
 W\Big(\frac{m \ell_1 - n_1 p}{Z}, \cdot\Big)
  W\Big(\frac{m \ell_2 - n_2 p}{Z}, \cdot\Big).
\end{equation}
\end{lemma}

\subsection{Poisson summation}

Now we focus on  $T(N, \mathcal{A}_j)$ and $V$.  
Applying Poisson summation in $m$ modulo $c_1 c_2$, we have
\begin{equation}
\label{eq:VPoissonDef}
    V = 
    \sum_{m} \widehat{S}(m) \cdot \widehat{I}(m),
\end{equation}
where
\begin{equation}
     \widehat{S}(m) = \frac{1}{c_1 c_2} \sum_{x \shortmod{c_1 c_2}} 
    S(0, p n_1 - x \ell_1;c_1) S(0, pn_2 - x \ell_2;c_2) e_{c_1 c_2}(xm),
\end{equation}
and 
\begin{equation}
    \widehat{I}(m) = 
     \int_{-\infty}^{\infty} 
     W\Big(\frac{x \ell_1 - n_1 p}{Z}, \cdot\Big)
  W\Big(\frac{x \ell_2 - n_2 p}{Z}, \cdot\Big) e_{c_1 c_2}(-xm) dx.
\end{equation}

\begin{lemma}
\label{lemma:IhatProperties}
    The expression $\widehat{I}(m)$ is very small unless
    \begin{equation}
   \label{eq:IhatSupportTrivial}
    |\ell_2 n_1 - \ell_1 n_2| \ll  \frac{L^jZ}{p}.
\end{equation}
In addition, for some essentially $p^{\varepsilon}$-inert $w(\cdot)$, we have
\begin{equation}
\label{eq:IhatInert}
    \widehat{I}(m) 
    = \frac{Z}{\ell_2} e\Big(-\frac{m n_2 p}{\ell_2 c_1 c_2}\Big) w(m,\cdot),
\end{equation}
 and this is very small unless
\begin{equation}
\label{eq:Ihatdecay}
    |m| \ll p^{\varepsilon} \frac{L^jQ^2}{Z}.
\end{equation}
\end{lemma}
\begin{proof}
    Changing variables $x \rightarrow \frac{Zx + n_2 p}{\ell_2}$ shows
\begin{equation}
\label{eq:IhatAfterCOV}
    \widehat{I}(m) = \frac{Z}{\ell_2}
    e\Big(- \frac{m n_2 p}{\ell_2 c_1 c_2} \Big)
\int_{-\infty}^{\infty} 
     W(x, \cdot)
  W\Big(\frac{x \ell_1}{\ell_2 } +  \frac{(\ell_1 n_2 p - \ell_2 n_1 p)}{\ell_2 Z}, \cdot\Big) e_{\ell_2 c_1 c_2}(-xmZ) dx.    
\end{equation}
The integral is trivially very small unless \eqref{eq:IhatSupportTrivial} holds.
In addition, the standard integration by parts bound shows 
$\widehat{I}(m)$ is very small unless
\eqref{eq:Ihatdecay} holds.  Then with this condition in place, the exponential factor within the integrand in \eqref{eq:IhatAfterCOV} is essentially $p^{\varepsilon}$-inert. 
\end{proof}

Write $c_1 = c_{10} c_1'$ and $c_2 = c_{20} c_2'$ where $c_{10}$ and $c_{20}$ share all the same prime factors, and where $(c_1', c_2') = (c_1' c_2', c_{10} c_{20}) =  1$.  Note $(c_1, c_2) = (c_{10}, c_{20})$.
We also implicitly assume throughout that $(c_1, \ell_1) = (c_2, \ell_2) = 1$ since this assumption is in place via \eqref{eq:TNAdefLate}.
\begin{lemma}
\label{lemma:ShatEvaluation}
We have $\widehat{S}(m) = 0$ unless $(c_{10}, c_{20}) | m$ and $(m, c_1' c_2') = 1$.  
Under these assumptions, we have
\begin{equation}
    \widehat{S}(m) = 
     e_{c_1'}(\overline{c_{10} c_2 \ell_1} p n_1 m)
     e_{c_2'}(\overline{c_{20} c_1 \ell_2} p n_2 m)
      \mathop{\sumstar_{t_1 \shortmod{c_{10}}} \quad  \sumstar_{t_2 \shortmod{c_{20}}}}_{c_{20} t_1  + c_{10} t_2   \equiv m \shortmod{c_{10} c_{20}}} 
        e_{c_{10}}(\overline{c_1' c_2' \ell_1} t_1 p n_1) e_{c_{20}}(\overline{c_1' c_2' \ell_2} t_2 pn_2).
\end{equation}
In addition, $\widehat{S}(0) = 0$ unless $c_1 = c_2$ in which case
\begin{equation}
\label{eq:Shat0Evaluation}
\widehat{S}(0) = 
S (n_1 \ell_2 - n_2 \ell_1, 0;c_1).
\end{equation}
\end{lemma}
\begin{proof}
    By orthogonality of characters, we deduce
    \begin{equation}
        \widehat{S}(m) = 
       \mathop{\sumstar_{t_1 \shortmod{c_1}} \quad \sumstar_{t_2 \shortmod{c_2}}}_{-c_2 t_1 \ell_1 - c_1 t_2 \ell_2 + m \equiv 0 \shortmod{c_1 c_2}} 
        e_{c_1}(t_1 p n_1) e_{c_2}(t_2 pn_2).
    \end{equation}
This factors by the Chinese remainder theorem.  Consider the $c_1'$ factor of the sum.  Here $t_1$ is uniquely determined by $c_{2} t_1 \ell_1 \equiv m \pmod{c_1'}$, equivalently, $t_1 \equiv m \overline{c_2 \ell_1} \pmod{c_1'}$.  However, the condition that $(t_1, c_1) = 1$ then additionally forces $(m, c_1') = 1$.
Hence this factor evaluates as
\begin{equation}
\label{eq:c1'sum}
    e_{c_1'}(\overline{c_{10} c_2 \ell_1} p n_1 m).
\end{equation}
By symmetry, and assuming $(m, c_2') = 1$,
the factor of modulus $c_2'$ equals
\begin{equation}
\label{eq:c2'sum}
    e_{c_2'}(\overline{c_{20} c_1 \ell_2} p n_2 m).
\end{equation}

Finally, consider the sum modulo $c_{10} c_{20}$.  Here we change variables $t_1 \rightarrow \overline{c_2' \ell_1} t_1$ and likewise $t_2 \rightarrow \overline{c_1' \ell_2} t_2$.  This factor then evaluates as
\begin{equation}
\label{eq:c10c20sum}
     \mathop{\sumstar_{t_1 \shortmod{c_{10}}} \quad  \sumstar_{t_2 \shortmod{c_{20}}}}_{c_{20} t_1  + c_{10} t_2   \equiv m \shortmod{c_{10} c_{20}}}
        e_{c_{10}}(\overline{c_1' c_2' \ell_1} t_1 p n_1) e_{c_{20}}(\overline{c_1' c_2' \ell_2} t_2 pn_2).
\end{equation}
Note that this sum is empty unless $(c_{10}, c_{20})$ divides $m$.

Now suppose $m=0$.  The previous work shows $\widehat{S}(0) = 0$ unless $c_1' = c_2' = 1$.  The congruence $c_{20} t_1 + c_{10} t_2 \equiv 0 \pmod{c_{10} c_{20}}$ then implies $c_{20} t_1 \equiv 0 \pmod{c_{10}}$, and similarly $c_{10} t_2 \equiv 0 \pmod{c_{20}}$.  Since $(t_1, c_{10}) = (t_2, c_{20}) = 1$, this implies $c_{10}$ divides $c_{20}$ and vice-versa.  Hence $c_{10} = c_{20}$.
The claimed formula for $\widehat{S}(0)$ now follows directly.
\end{proof}

We also record the following alternative formula for $\widehat{S}(m)$.
As shorthand notations, we let
\begin{equation}
\alpha_{\infty}
= \alpha_{\infty}(c_1, c_2, \ell_2, p, n_2, m)
= 
e_{c_1 c_2 \ell_2}(p n_2 m),
\end{equation}
and
\begin{align}
\label{eq:alphac0part}
\alpha_0(m)
=  
 e_{c_{10} c_{20}}(-\overline{c_1' c_2' \ell_2} p n_2 m)
  \mathop{\sumstar_{t_1 \shortmod{c_{10}}} \quad  \sumstar_{t_2 \shortmod{c_{20}}}}_{c_{20} t_1  + c_{10} t_2   \equiv m \shortmod{c_{10} c_{20}}}
        e_{c_{10}}(\overline{c_1' c_2' \ell_1} t_1 p n_1) e_{c_{20}}(\overline{c_1' c_2' \ell_2} t_2 pn_2).
\end{align}
\begin{lemma}
\label{lemma:ShatEvaluation2}
For $(m, c_1' c_2') = 1$ 
we have
\begin{equation}
    \widehat{S}(m) = 
    \alpha_{\infty} 
    \cdot
    \alpha_0(m)
    \cdot 
    e_{c_1'}(\overline{c_{10} c_2 \ell_1} p n_1 m)
   e_{c_1' \ell_2 }(-\overline{c_{10} c_{2} } p n_2 m ).
\end{equation}
\end{lemma}
\begin{proof}
    We use the elementary reciprocity formula repeatedly on \eqref{eq:c2'sum}, giving
\begin{align}
\label{eq:c2'sumEvalReciprocity}
     e_{c_2'}(\overline{c_{20} c_1 \ell_2} p n_2 m)
     = e_{c_1' c_{10} c_{20} \ell_2}(-\overline{c_2'} p n_2 m)
 e_{c_1 c_2 \ell_2}(p  n_2 m)
 \nonumber\\
 = 
 e_{c_{10} c_{20}}(-\overline{c_1' c_2' \ell_2} p n_2 m)
e_{c_1' \ell_2}(- \overline{c_{2} c_{10}} p n_2 m)
\cdot \alpha_{\infty}
 .
\end{align}
Applying  \eqref{eq:c2'sumEvalReciprocity} into
Lemma \ref{lemma:ShatEvaluation} gives the claimed formula.
\end{proof}

Next we refine our evaluation of $\widehat{S}(m)$.
\begin{lemma}
\label{lemma:ShatEvaluationSimplified2}
Suppose that $(m, c_1' c_2') = 1$ and let $D=\ell_2 n_1 - \ell_1 n_2.$ Then 
\begin{align}
 \widehat{S}(m) = 
 \alpha_{\infty} 
    \cdot
    \alpha_0(m)
    \cdot 
 \begin{cases}
 \label{eq:ShatmEvaluationelliDistinct}
    e_{c_1' \ell_2}(\overline{c_{10} c_2 \ell_1} pm D), \quad &\text{if} \quad (\ell_1, \ell_2)=1\\
     e_{\ell_2}(-\overline{c_{1} c_{2} } p n_2 m)
     \cdot
    e_{c_1'}(\overline{c_{10} c_2 \ell_1 \ell_2} pm D), \quad &\text{if} \quad   (\ell_1, \ell_2)>1.
\end{cases}
\end{align}
\end{lemma}
\begin{proof}
Consider the case when $(\ell_1, \ell_2) = 1$ first. We write
 $e_{c_1'}(\overline{c_{10} c_2 \ell_1} p n_1 m)
 = e_{c_1' \ell_2}(\overline{c_{10} c_2 \ell_1} p m \ell_2 n_1)$ and
$e_{c_1' \ell_2 }(-\overline{c_{10} c_{2} } p n_2 m )
    =  e_{c_1' \ell_2 }(-\overline{c_{10} c_{2} \ell_1 } p m \ell_1 n_2 ),$ which  
  gives the desired formula.
    
    Next suppose $(\ell_1, \ell_2) > 1$.  This implies that $(\ell_2, c_1') = 1$ since $\ell_1$ and $\ell_2$ share a unique prime factor, and since $(\ell_1, c_1) = 1$ by assumption.  Then we write by the Chinese remainder theorem
       $ e_{c_1' \ell_2 }(-\overline{c_{10} c_{2} } p n_2 m)
        = e_{c_1'}(-\overline{c_{10} c_{2} \ell_2 } p n_2 m)
        e_{\ell_2}(-\overline{c_{1} c_{2} } p n_2 m).$
 The formula follows by combining the additive characters of modulus $c_1'$ via
 $e_{c_1'}(\overline{c_{10} c_2 \ell_1} p n_1 m)
     e_{c_1'}(-\overline{c_{10} c_{2} \ell_2 } p n_2 m)
  = e_{c_1'}(\overline{c_{10} c_2 \ell_1 \ell_2} p m D).$
\end{proof}

\subsection{Zero frequency}
Let $V_0$ correspond to the term with $m=0$ in \eqref{eq:VPoissonDef}, and let $T_0(N, \mathcal A_j)$ denote the contribution of $V_0$ to $T(N, \mathcal{A}_j)$ in \eqref{eq:TNAdefLate}.
Similarly, let $V_{\neq 0} := V - V_0$ and $T_{\neq 0}:=T_{\neq 0}(N, \mathcal A_j) = T(N, \mathcal{A}_j) - T_0(N, \mathcal A_j)$.
\begin{lemma}
\label{lemma:T0'bound}
 For $j=1,2$, we have 
\begin{equation}
\label{eq:T0'bound}
    T_0(N, \mathcal A_j) \ll 
    p^\varepsilon \frac{Q^2 L^{1+\frac{5j}{2}} p}{d^4 \sqrt{N}}, 
\end{equation}
and its contribution to $S(N, \mathcal{A}_j)$ via \eqref{eq:SboundInTermsofT} is bounded by 
\begin{equation}
    \ll p^{\varepsilon} N^{1/4} p^{1/2} L^{\frac12 + \frac{3j}{4}}.
\end{equation}
\end{lemma}
\begin{proof}

Using \eqref{eq:IhatSupportTrivial}, \eqref{eq:IhatInert}, and \eqref{eq:Shat0Evaluation}, together with
 $|S(n,0;c)| \leq (n,c)$, we obtain
\begin{equation}
    V_0 \ll \frac{Z }{\ell_2} \delta(c_1 = c_2)
\cdot
    (n_1 \ell_2 - n_2 \ell_1, c_1) 
    \cdot 
    \delta\Big(1\leq |\ell_1 n_2 - \ell_2 n_1| \ll \frac{L^j Z}{p}\Big) + O(p^{-100}).
\end{equation}
Therefore we arrive at
\begin{align}
    T_0(N, \mathcal{A}_j)    
\ll 
\sum_{\substack{\ell_1, \ell_2 \in \mathcal{L}_j \\ n_1, n_2 \asymp N' \\ 
1 \leq |\ell_1 n_2 - \ell_2 n_1| \ll \frac{L^j Z}{p}}} 
\Big(|b_j(\ell_1)\lambda_g(n_2)|^2+|b_j(\ell_2)\lambda_g(n_1)|^2 \Big)
\sum_{c_1\asymp Q}(n_1\ell_2-n_2\ell_1, c_1)
\frac{Z}{\ell_2}.
\end{align}
Since we have by assumption $\ell_1 n_2 \neq \ell_2 n_1$, we can use \eqref{eq:gcdsum} to handle the sum over $c_1$. 
The contribution to $T_0(N, \mathcal A_j)$ from these terms is bounded by
$\ll p^{-1+\varepsilon} L N' Q Z^2$.
Using \eqref{eq:M'N'over1plusPbound} for $N'$, \eqref{eq:ZboundUnified} for $Z$, and
$Q \ll \frac{\sqrt{N L^j}}{d}$ we obtain \eqref{eq:T0'bound}.
The contribution of \eqref{eq:T0'bound} to $S''(N, \mathcal{A}_j)$ is then bounded by
\begin{equation*}
    \sum_d 
    p^{\varepsilon} \frac{\sqrt{N}}{Q \sqrt{L^j}}
    \Big(\frac{Q^2 L^{1+\frac{5j}{2}} p}{d^4 \sqrt{N}}\Big)^{1/2}
    \ll p^{\varepsilon} N^{1/4} p^{1/2} L^{\frac12 + \frac{3j}{4}}.  \qedhere
\end{equation*}
 \end{proof}

\subsection{Non-zero frequency}
We consider the contribution from $m\not=0$. 
Recall that
\begin{align}
  T_{\neq 0} =  
   \sum_{\substack{\ell_1, \ell_2 \in \mathcal{L}_j \\ n_1, n_2 \asymp N' \\ 1\leq  |\ell_2 n_1 - \ell_1 n_2| \ll \frac{L^j Z}{p}}}
\sum_{\substack{(c_1, \ell_1) = 1 \\ (c_2, \ell_2) = 1 \\ c_1 \asymp c_2 \asymp Q}} 
b_j(\ell_1)\overline{b_j(\ell_2)}\chi(c_1) \overline{\chi}(c_2)
\lambda_{g}(n_1) \lambda_{g}(n_2)
 \sum_{0 < |m| \ll \frac{L^j Q^2p^\varepsilon}{Z}} \widehat{S}(m) \widehat{I}(m)
\end{align}
up to a very small error term,
and that $\widehat{S}(m)$ is given by Lemma \ref{lemma:ShatEvaluation}  (and alternatively, Lemma \ref{lemma:ShatEvaluation2}). Now, further decompose this as
\begin{equation}
    T_{\neq 0}
    = \underbrace{T_{\neq 0}''}_{\ell_1 \neq \ell_2}  +\underbrace{T_{\neq 0}^{\text{diag}}}_{\ell_1 = \ell_2}
     + O(p^{-100}).
\end{equation}

We shall focus on $T_{\neq 0}''$ since it is the harder case.  In this case, we use \eqref{eq:ShatmEvaluationelliDistinct} to evaluate $\widehat{S}(m)$, and \eqref{eq:IhatInert} for the evaluation of $\widehat{I}(m)$.  Note the cancellation of the factor $\alpha_{\infty}$ in these two terms.  Hence
\begin{align}
\label{eq:Tneq0''formula}
   T_{\neq 0}'' =&\  
  \sum_{\substack{\ell_1 \neq \ell_2 \in \mathcal{L}_j \\ n_1, n_2 \asymp N' \\ D=\ell_2 n_1 - \ell_1 n_2\neq0 \\ |D| \ll \frac{L^j Z}{p}}}
\sum_{\substack{(c_1, \ell_1) = 1 \\ (c_2, \ell_2) = 1 \\ c_1 \asymp c_2 \asymp Q}}
b_j(\ell_1)\overline{b_j(\ell_2)}\chi(c_1) \overline{\chi}(c_2)
\lambda_{g}(n_1) \lambda_{g}(n_2)
\frac{Z}{\ell_2}
\nonumber\\
&\ \times \sum_{\substack{c_{10} c_1' = c_1 \\ c_{20} c_2' = c_2 \\ (\dots)}}
\sum_{0 < |m| \ll \frac{L^j Q^2p^\varepsilon}{Z}} 
\alpha_0(m)
    \cdot 
    e_{c_1' \ell_2}(\overline{c_{10} c_2 \ell_1} pm D)
    w(m,\cdot),
\end{align}
where the $(\dots)$ in the summation over $c_{10} c_1' = c_1$ etc. 
is shorthand for the conditions that $(c_{10} c_{20}, c_1' c_2') =1$ and that $c_{10}$ and $c_{20}$ share all the same prime factors.
We similarly record
\begin{align}
\label{eq:Tneq0diagFormula}
   T_{\neq 0}^{\text{diag}} = &\ 
  \sum_{\substack{\ell_1 = \ell_2 \in \mathcal{L}_j \\ n_1, n_2 \asymp N' \\ 1\leq  |n_1 - n_2| \ll \frac{Z}{p}}}
\sum_{\substack{(c_1, \ell_1) = 1 \\ (c_2, \ell_2) = 1 \\ c_1 \asymp c_2 \asymp Q}}
|b_j(\ell_1)|^2\chi(c_1) \overline{\chi}(c_2)
\lambda_{g}(n_1) \lambda_{g}(n_2)
\frac{Z}{\ell_2}
\sum_{\substack{c_{10} c_1' = c_1 \\ c_{20} c_2' = c_2 \\ (\dots)}}\nonumber
\\
& \ \times \sum_{0 < |m| \ll \frac{L^j Q^2p^\varepsilon}{Z}} 
\alpha_0(m)
    \cdot 
    e_{\ell_2}(-\overline{c_{1} c_{2} } p n_2 m)
     \cdot
    e_{c_1'}(\overline{c_{10} c_2 \ell_1} pm (n_1-n_2))
    w(m,\cdot).
\end{align}

\subsection{Separation of variables}
The next step is to separate the variables $c_1'$ and $c_2'$ inside $\alpha_0(m)$, which was defined in \eqref{eq:alphac0part}.  Note that the expression \eqref{eq:alphac0part}, as a function of $c_1'$ and $c_2'$, depends only on the product $c_1' c_2'$.  Therefore, by Fourier/Mellin decomposition, we have
\begin{equation}
\label{eq:alphac10c20Fourier}
    \alpha_0(c_1'c_2', m, \cdot)
    = 
    \sum_{\psi \shortmod{c_{10} c_{20}}} \widehat{\alpha}(\psi,m, \cdot) \psi(c_1' c_2'),
\end{equation}
where
\begin{align}
    \widehat{\alpha}(\psi,m,\cdot)
    = &\ \frac{1}{\varphi(c_{10} c_{20})}
    \sum_{x \shortmod{c_{10} c_{20}}} \overline{\psi(x)}
    e_{c_{10} c_{20}}(-\overline{x \ell_2} p n_2 m)
 \nonumber\\
 &\ \times \mathop{\sumstar_{t_1 \shortmod{c_{10}}}\quad \sumstar_{t_2 \shortmod{c_{20}}}}_{c_{20} t_1  + c_{10} t_2   \equiv m \shortmod{c_{10} c_{20}}} 
        e_{c_{10}}(\overline{x \ell_1} t_1 p n_1) e_{c_{20}}(\overline{x \ell_2} t_2 pn_2).
\end{align}
Here, $\widehat{\alpha}(\psi, \cdot)$ depends on $\psi$, $c_{10}$, $c_{20}$, $\ell_1$, $\ell_2$, $m$, $p$, $n_1$, $n_2$.

\begin{lemma}
\label{lemma:alphaHatFourierTransformL1bound}
Let $D = \ell_2 n_1 - \ell_1 n_2$. We have $\hat{\alpha}(\psi, m)$ is nonvanishing unless $(c_{10}, c_{20})\mid m$ in which case write $m=(c_{10},c_{20})m_0'm'$ where $(m', c_{10}c_{20})=1$ and $ m_0'\mid (c_{10}c_{20})^\infty$. 
Then
\begin{equation}\label{factorhatpsi}
\widehat{\alpha}(\psi,m, \cdot)=\overline{\psi}\big(m'\big)\widehat{\alpha}(\psi,m_0', \cdot)
\end{equation}
and
    \begin{equation}
    \label{eq:alphaHatFourierTransformL1bound}
       \|\widehat{\alpha}(\psi, m_0')\|_1=\sum_{\psi}|\widehat{\alpha}(\psi, m_0', \cdot)| \leq p^{\varepsilon}(D, c_{10}, c_{20})
        \prod_{q: v_q(c_{10}) = v_q(c_{20})} q^{1/2}.    \end{equation}
\end{lemma}
\begin{proof}
From Lemma \ref{lemma:ShatEvaluation}, we have that $\widehat{\alpha}(\psi)$ vanishes unless $(c_{10}, c_{20})$ divides $m$, in which case we can write $m=m_0m'$ where $(m', c_{10}c_{20})=1$ and $(c_{10}, c_{20})\mid m_0\mid (c_{10}c_{20})^\infty$. 
After writing $\psi$ as a product of characters with prime powers moduli, we see that the sum to be bounded factors over prime  powers by the Chinese remainder theorem, as does the claimed bound, so it suffices to show it one prime at a time. Let $q$ be a prime dividing $c_{10}$ and without loss of generality, suppose that $q_{10}:=\nu_q(c_{10})\leq \nu_q(c_{20})=:q_{20}$.  Then from Lemma \ref{lemma:ShatEvaluation}, we have that $\widehat{\alpha}(\psi)$ vanishes unless $q_{10}$ divides $m_0$. Note: $m_0'=m_0/q_{10}$.   

Observe that the congruence in the sum over $t_1$ and $t_2$ is equivalent to $\frac{q_{20}}{q_{10}} t_1 + t_2 \equiv m_0'm' \pmod{q_{20}}$.  This determines $t_2$ uniquely, however one must retain the condition $(t_2, q) = 1$ by $(m_0'm' - \frac{q_{20}}{q_{10}} t_1, q) = 1$.  Applying the above observations, and simplifying the sum, we obtain
    \begin{equation}
    \widehat{\alpha}(\psi,m, \cdot)
    = \frac{1}{\varphi(q_{10}q_{20})}
    \sum_{x \shortmod{q_{10}q_{20}}} \overline{\psi(x)}
 \sumstar_{\substack{t_1 \shortmod{q_{10}} \\ (m_0'm' - \frac{q_{20}}{q_{10}} t_1, q) = 1} } 
 e_{q_{10}}(\overline{x \ell_1 \ell_2} t_1 p D).
\end{equation}

Suppose that $q_{10}<q_{20}$.  In this case, the condition $(m_0'm' - \frac{q_{20}}{q_{10}} t_1, q) = 1$ is equivalent to $(m_0', q) = 1$, that is $m_0'=1$.  The inner sum over $t_1$ above then evaluates as $S(D, 0;q_{10})$, which is independent of $x$.  The sum over $x$ then vanishes unless $\psi$ is trivial.  Hence in this case,
\begin{equation}
    \widehat{\alpha}(\psi,m,\cdot) =
    \delta_{\mathrm{cond}(\psi) = 1} S(D,0;q_{10})=\overline{\psi}(m')\widehat{\alpha}(\psi,1, \cdot)=\overline{\psi}(m')\widehat{\alpha}(\psi,m_0, \cdot),
\end{equation}
and it is easy to see that \eqref{eq:alphaHatFourierTransformL1bound} holds using $|S(D,0;c)| \leq (D,c)$.

Now, suppose that $q_{10} = q_{20} =: q_0$, say. In this case, we change variables $x \rightarrow t_1 x$ to obtain
 \begin{equation}
    \widehat{\alpha}(\psi,m, \cdot)
    = \frac{1}{\varphi(q_{0}^2)}
    \sum_{x \shortmod{q_{0}^2}} \overline{\psi(x)}
    e_{q_{0}}(\overline{x \ell_1 \ell_2}  p D)
 \sumstar_{\substack{t_1 \shortmod{q_{0}} \\ (m_0'm' - t_1, q) = 1} } 
 \overline{\psi}(t_1)
.
\end{equation}
The inner sum over $t_1$ then takes the form
\begin{equation}
    \sumstar_{\substack{t_1 \shortmod{q_{0}} \\ (m_0'm' - t_1, q) = 1} } 
 \overline{\psi}(t_1)
 = \sum_{t_1 \shortmod{q_0}}
 \overline{\psi}(t_1)
 - \sum_{r \shortmod{q_0/q}} \overline{\psi}(m_0'm' + qr).
\end{equation}
The first sum on the right hand side evaluates as $\varphi(q_0) \delta_{\mathrm{cond}(\psi)=1}=\varphi(q_0)\delta_{\mathrm{cond}(\psi)=1}\overline{\psi}(m')$.  
The second sum on the right hand side evaluates as
\begin{equation}
    \frac{q_0}{q} \overline{\psi}(m') \overline{\psi}(m_0')\delta_{\mathrm{cond}(\psi) \leq q}.
\end{equation}
Therefore, $\widehat{\alpha}(\psi, m,\cdot)=\overline{\psi}(m')\widehat{\alpha}(\psi, m_0',\cdot)$ and
\begin{equation}
\widehat{\alpha}(\psi, m_0',\cdot)=\frac{1}{\varphi(q_0)}\sum_{x\bmod {q_0}}\overline{\psi(x)} e_{q_0}(\overline{x\ell_1\ell_2}pD) \Big(\varphi(q_0) \delta_{\mathrm{cond}(\psi)=1} -\frac{q_0}{q}\overline{\psi}(m_0')\delta_{\text{cond}(\psi)\leq q}\Big).
\end{equation}
If $\mathrm{cond}(\psi)=1$, the sum over $x$ is a Ramanujan sum, and we have
\begin{equation}
    \widehat{\alpha}(\psi_0, m_0', \cdot) = S(D, 0;q_0) (1 - \frac{\overline{\psi}(m_0')}{q-1}).
\end{equation}
If $\mathrm{cond}(\psi) = q$, then we have
\begin{equation}
    \widehat{\alpha}(\psi, m_0', \cdot) = - \frac{\overline{\psi}(m_0')}{q-1}  
    \sum_{x \shortmod{q_{0}}} \psi(x)
    e_{q_{0}}(x\overline{ \ell_1 \ell_2}  p D).
\end{equation}
The sum over $x$ vanishes unless $\frac{q_0}{q}$ exactly divides $D$, in which case
\begin{equation}
    \widehat{\alpha}(\psi, m_0', \cdot) = - \frac{\psi(\ell_1 \ell_2) \overline{\psi}(p  \frac{D}{q_0/q})\overline{\psi}(m_0') }{q-1} \frac{q_0}{q}
    \tau(\psi),
\end{equation}
whence $|\widehat{\alpha}(\psi)| \ll q^{-1/2} (D, q_0)$ since $(D, q_0) = \frac{q_0}{q}$ in this case.  The claimed bound follows easily after suming over $\psi$.
\end{proof}

To treat $T_{\neq 0}^{\text{diag}}$ we will also need to separate variables in $e_{\ell_2}(- \overline{c_1 c_2} p n_2 m)$, which appears in \eqref{eq:Tneq0diagFormula}.
\begin{lemma}
\label{lemma:exponentialmodulusell2SeparationofVariables}
    Suppose $(\ell_2, n_2 m) = 1$.  Then
    \begin{equation}
        e_{\ell_2}(- \overline{c_1 c_2} p n_2 m)
        = \frac{1}{\varphi(\ell_2)} \sum_{\eta \shortmod{\ell_2}} \tau(\overline{\eta}) \eta(-\overline{c_1 c_2} p n_2 m ).
    \end{equation}
\end{lemma}

Next we briefly discuss the
 cases with $(\ell_2, nm) \neq 1$.
If $\ell_2$ divides $n_2 m$ then the exponential factor is identically $1$, and the variables are already separated here.  If $\ell_2 = \nu^2$ with $\nu$ prime and $\nu$ exactly divides $n_2 m$, then we apply a similar formula with modulus $\nu$ instead of $\nu^2$, which makes the forthcoming estimations even better.

We also have an Archimedean separation of variables, which is also standard.  
For instance, by Fourier analysis, we have
\begin{equation}
\label{eq:archimedeanseparationofvariables}
    w(\cdot) = w(c_1, c_2, m, \cdot)
    = \int_{\mr^3} \widehat{w}(x_1, x_2, y, \cdot) e(c_1 x_1 + c_2 x_2 + ym) dx_1 dx_2 dy,
\end{equation}
where $\int_{\mr^3} |\widehat{w}(x_1, x_2, y, \cdot)| dx_1 dx_2 dy \ll p^{\varepsilon}$.  This separates the variables $c_1, c_2$, and $m$ inside $w$ at essentially no cost.
For this estimate, we used that $w$ is essentially $p^{\varepsilon}$-inert in all variables.

\subsection{Applying trilinear sums of Kloosterman fractions}
\begin{lemma}
\label{lemma:Tneq0''bound}
    For $j=1,2$, and $\sigma = \frac{1}{20}$, we have for $L\ll p^{1/6-\varepsilon}$
\begin{equation}
\label{eq:Tneq0''finalbound}
    T_{\neq 0}''  \ll d^{-2} p^{3\sigma+\varepsilon}L^{5j/2+2-5j\sigma/2}N^{-1/2-3\sigma/2}Q^{5-4\sigma},
\end{equation}
and its corresponding contribution to $S''(N, \mathcal{A}_j)$ via \eqref{eq:SboundInTermsofT} is 
\begin{equation}
    \ll p^{\varepsilon} N^{1-7\sigma/4} p^{3\sigma/2}L^{3j/2+1-9j\sigma/4}.
\end{equation}
\end{lemma}
\begin{proof}
We pick up from \eqref{eq:Tneq0''formula} and apply \eqref{eq:alphac10c20Fourier}, \eqref{factorhatpsi}, and \eqref{eq:archimedeanseparationofvariables}. Let $K=\frac{L^jQ^2p^{\varepsilon}}{Z}$ and write $c_{0}=(c_{10}, c_{20})$ and $m=c_0 m_0'm'$ where $m_0'\mid c_0^\infty$ and $(m', c_0)=1$ and write $c_{10}=c_{0}c_{10}'$,  $c_{20}=c_{0}c_{20}'$.
Using the notation $c_{10} \sim c_{20}$ to denote that these integers share all the same prime factors,
we write 
\begin{align}
   T_{\neq 0}'' =
  \int_{\mr^3} 
  \sum_{\substack{\ell_1 \neq \ell_2 \in \mathcal{L}_j \\ n_1, n_2 \asymp N' \\ D=\ell_2 n_1-\ell_1 n_2\neq0 \\ |D| \ll \frac{L^j Z}{p}}}
\sum_{\substack{(c_{10}, \ell_1) = 1 \\ (c_{20}, \ell_2) = 1 \\ c_{10}, c_{20} \ll Q \\ c_{10} \sim c_{20} }}\sum_{\substack{m_0'\mid c_{0}^\infty\\ m_0'\ll \frac{K}{c_0}}}
b_j(\ell_1)\overline{b_j(\ell_2)}\chi(c_{10}) &\overline{\chi}(c_{20})
\lambda_{g}(n_1) \lambda_{g}(n_2)
\frac{Z}{\ell_2}\\
&\times\sum_{\substack{\psi}}  
\widehat{w}(\cdot)
\widehat{\alpha}(\psi, m_0',\cdot)U(\cdot),
\label{eq:Tneq0''formula2}
\end{align}
where $U= U(\cdot)$ depends on all outer variables is given by
\begin{align}
    U(\cdot) =  \sum_{\substack{1\leq |m'|\ll K/c_0m_0'\\ (c_0, m')=1}}\sum_{\substack{(c_1', c_2') = 1 \\ (c_1', \ell_1) = (c_2', \ell_2) = 1 \\ c_1' \asymp \frac{Q}{c_{10}}, \,c_2' \asymp \frac{Q}{c_{20}}}} 
 \chi \psi(c_1') \overline{\chi} \psi(c_2') \overline{\psi}{(m')})   
     & e_{c_1' \ell_2}(\overline{c_0c_{10}' c_{20}'  \ell_1c_2'} pm_0'm' D)
   \\ & \times e(x_1 c_{10} c_1') e(x_2 c_{20} c_2') e(yc_0m_0'm').
\end{align} 
We are now ready to apply Theorem \ref{thm:BC}.  We view the outer variables $\ell_1$, $\ell_2$, $n_1$, $n_2$, $c_{10}$, $c_{20}, m_0'$ and $\psi$ as fixed. 
Change variables $q_1 = c_1' \ell_2$ and $q_2 = c_{0}c_{10}' c_{20}' \ell_1 c_2'$. Then we have  $(q_1, q_2) = 1$ since $(c_{10}c_1', \ell_1)=(c_{20}c_2', \ell_2)=(c_1', c_2')=1$ and $c_{10}'c_{20}'\mid c_0^\infty$.
Hence we can express $U$ as a trilinear form:
\begin{equation}
    U(\cdot) =\sum_{m'} \sum_{(q_1, q_2) = 1} 
    \alpha_{q_1} \beta_{q_2}\gamma_{m'}
    e_{q_1}(\overline{q_2} pm'D),
\end{equation}
where $q_1 \asymp \frac{Q \ell_2}{c_{0}c_{10}'}$, $q_2 \asymp Q \ell_1 c_{10}'$ and $1\leq |m'|\ll \frac{K}{c_0m_0'}$, and the coefficients $\alpha$, $\beta$ and $\gamma$ have their sup-norms bounded by $1$ and encode the various side conditions such as $q_1 \equiv 0 \pmod{\ell_2}$, $q_2 \equiv 0 \pmod{c_{0}c_{10}' c_{20}' \ell_1}$, and $(m', c_0c_1'c_2')=1$ from Lemma \ref{lemma:ShatEvaluation}.
Note that for fixed $c_{0}, c_{10}', c_{20}'$, we have 
\begin{equation}\label{abgbound}
\| \alpha \| \ll  \Big(\frac{Q }{c_{0}c_{10}'} \Big)^{1/2}
,
\quad
\| \beta \| \ll \Big(\frac{Q }{c_{0}c_{20}'} \Big)^{1/2}, 
\quad \|\gamma\|\ll  \Big(\frac{K}{c_0m_0'}\Big)^{1/2}.
\end{equation}
Also, note for simplification that the factor $\frac{|a| K}{MN}$ in Theorem \ref{thm:BC} here takes the form
\begin{equation}
    \frac{p \cdot |D| \cdot K/c_0 }{\frac{Q \ell_2}{c_{10}} \frac{Q \ell_1 c_{10}}{c_0}}
    \ll p^{\varepsilon} \frac{p \cdot \frac{L^j Z}{p} \cdot \frac{L^j Q^2}{Z}}{Q^2 L^{2j}}= p^{\varepsilon}.
\end{equation}
Furthermore, the factor $(MN)^{7/20} (M+N)^{1/4}$ in Theorem \ref{thm:BC} here takes the form
\begin{equation}
    \Big(\frac{Q L^j}{c_0 c_{10}'} \cdot Q L^j c_{10}'\Big)^{7/20}
    \Big(\frac{Q L^j}{c_0 c_{10}'} + Q L^j c_{10}'\Big)^{1/4}
    \ll \frac{(Q L^j)^{\frac{19}{20}} (c_{10}')^{1/4}}{c_0^{7/20}},
\end{equation}
and similarly the factor $(MN)^{3/8} (M+N)^{1/8}$ takes the form
\begin{equation}
    \Big(\frac{Q L^j}{c_0 c_{10}'} \cdot Q L^j c_{10}'\Big)^{3/8}
    \Big(\frac{Q L^j}{c_0 c_{10}'} + Q L^j c_{10}'\Big)^{1/8}
    \ll \frac{(Q L^j)^{\frac{7}{8}} (c_{10}')^{1/8}}{c_0^{3/8}}.
\end{equation}
Finally, we observe that since $Z \gg \frac{Q^2 p}{N}$, we have
$
    K \ll p^{\varepsilon} \frac{L^j Q^2}{Z} \ll \frac{L^j N}{p^{1-\varepsilon}}, 
$
and since $N \ll p^{1+\varepsilon}$, we deduce $K \ll L^j p^{\varepsilon}$.
Now Theorem \ref{thm:BC} implies
\begin{equation}
    U(\cdot)
    \ll \frac{p^{\varepsilon} QK^{1/2}}{(m_0'c_0^3 c_{10}'c_{20}')^{1/2}} 
    \Big( \big(\frac{K}{c_0m_0'}\big)^{7/20} \frac{(Q L^j)^{\frac{19}{20}} (c_{10}')^{1/4}}{c_0^{7/20}}
    + \big(\frac{K}{c_0m_0'}\big)^{1/2} \frac{(Q L^j)^{\frac{7}{8}} (c_{10}')^{1/8}}{c_0^{3/8}} \Big),
\end{equation}
which can be simplified using crude bounds in terms of the $c_0$-variables as
\begin{equation}
    U(\cdot)
    \ll \frac{p^{\varepsilon} QK^{1/2}}{m_0'^{17/20}c_0^{7/10} (c_0^3 c_{20}' \sqrt{c_{10}'})^{1/2}} 
    \Big( K^{7/20} (Q L^j)^{\frac{19}{20}}
    + K^{1/2} (Q L^j)^{\frac{7}{8}} \Big).
\end{equation}
Let $\sigma = 1/20$.  Then we have
\begin{equation}
    U(\cdot)
    \ll \frac{p^{\varepsilon} QK }{m_0'^{1-3\sigma}c_0^{7/10} (c_0^3 c_{20}' \sqrt{c_{10}'})^{1/2}} 
    \Big( K^{-3 \sigma} (Q L^j)^{1-\sigma}
    +  (Q L^j)^{\frac{7}{8}} \Big).
\end{equation}
Substitute this into \eqref{eq:Tneq0''formula2}.  Using Lemma \ref{lemma:alphaHatFourierTransformL1bound} and summing over $\psi$, we obtain
\begin{multline}
    T_{\neq 0}'' \ll
     \sum_{\substack{(c_{10}, \ell_1) = 1 \\ (c_{20}, \ell_2) = 1 \\ c_{10}, c_{20} \ll Q \\ c_{10} \sim c_{20}\\ c_0=(c_{10}, c_{20})}} 
     \sum_{\substack{\ell_1 \neq \ell_2 \in \mathcal{L}_j \\ n_1, n_2 \asymp N'\\ 1 \leq |\ell_2 n_1-\ell_1 n_2| \ll \frac{L^j Z}{p}}}
     \sum_{\substack{m_0'\mid c_0^\infty\\ m_0'\ll K/c_0}}
     (|b_j(\ell_1) \lambda_g(n_2)|^2 + |b_j(\ell_2) \lambda_g(n_1)|^2)
     \\  
     \times
     \frac{Z}{L^j}
    \|\widehat{\alpha}(\psi,m_0') \|_1
  \frac{p^{\varepsilon} QK }{m_0'^{1-3\sigma}c_0^{7/10} (c_0^3 c_{20}' \sqrt{c_{10}'})^{1/2}} 
    \Big( K^{-3 \sigma} (Q L^j)^{1-\sigma}
    +  (Q L^j)^{\frac{7}{8}} \Big),
\end{multline}
where  $c_0 = (c_{10}, c_{20})$.  Lemma \ref{lemma:alphaHatFourierTransformL1bound} implies
$\| \widehat{\alpha}(\psi,m_0') \|_1 \ll (D, c_0) c_0^{1/2}$. 
Note: by symmetry that both terms, with $|b_j(\ell_1) \lambda_g(n_2)|^2$, and $|b_j(\ell_2) \lambda_g(n_1)|^2$, give the same bound.

The sums over the $c$-variables and $m_0'$ can be estimated using using \eqref{eq:gcdsum} so that
\begin{equation}
    \sum_{\substack{(c_{10}, \ell_1) = 1 \\ (c_{20}, \ell_2) = 1 \\ c_{10}, c_{20} \ll Q \\ c_{10} \sim c_{20}}} \sum_{\substack{m_0'\mid c_0^\infty \\ m_0'\ll K/c_0}}
    \frac{(D, c_0)}{m_0'^{17/20}c_0^{7/10} (c_0^2 c_{20}' \sqrt{c_{10}'})^{1/2}} 
    \ll |D|^{\varepsilon} \sum_{c_0} c_0^{-17/10} 
    \sum_{m_0' c_{10}' c_{20}' \mid c_0^{\infty}} (c_{20}')^{-1/2} (c_{10}')^{-1/4}m_0'^{-17/20}
     \ll |D|^{\varepsilon}.
\end{equation}

Hence, we have
\begin{equation}
    T_{\neq 0}'' \ll \frac{Z}{L^j}
  p^{\varepsilon} QK
    \Big( K^{-3 \sigma} (Q L^j)^{1-\sigma}
    +  (Q L^j)^{\frac{7}{8}} \Big)
     \sum_{\substack{\ell_1 \neq \ell_2 \in \mathcal{L}_j \\ n_1, n_2 \asymp N'\\ 1 \leq |\ell_2 n_1-\ell_1 n_2| \ll \frac{L^j Z}{p}}}
     |b_j(\ell_1) \lambda_g(n_2)|^2
     .
\end{equation}
Using \eqref{RS} and \eqref{bjl2bound} and that the sum over $\ell_2$ and $n_1$ can be bounded   by
\begin{equation}
   \sum_{\ell_2 \in \mathcal{L}_j}
    \sum_{n_1 = \frac{\ell_2 n_1}{\ell_1} + O(\frac{L^j Z}{p \ell_1})} 1
    \ll \sum_{\ell_2 \in \mathcal{L}_j} \left(1 + \frac{L^j Z}{p \ell_1} \right)
    \ll L\left(1 + \frac{Z}{p}\right)
\end{equation}
uniformly in $\ell_1$ and $n_2$, we obtain 
\begin{equation}\label{ell2n1}
 \sum_{\substack{\ell_1 \neq \ell_2 \in \mathcal{L}_j \\ n_1, n_2 \asymp N'\\ 1 \leq |\ell_2 n_1-\ell_1 n_2| \ll \frac{L^j Z}{p}}}
     |b_j(\ell_1) \lambda_g(n_2)|^2
    \leq
    \sum_{\substack{\ell_1 \in \mathcal{L}_j \\ n_2 \asymp N'}} 
    |b_j(\ell_1) \lambda_g(n_2)|^2
    \sum_{\ell_2 \in \mathcal{L}_j}
    \sum_{n_1 = \frac{\ell_1 n_2}{\ell_2} + O(\frac{L^j Z}{p \ell_2})} 1\ll p^\varepsilon L N' L \left(1+\frac{Z}{p}\right). 
\end{equation}
Putting these bounds together, we deduce that
\begin{equation}\label{T''bound}
    T_{\neq 0}'' \ll
    p^{\varepsilon} L N'  L\Big(1 + \frac{Z}{p}\Big)
    \frac{Z}{L^j} 
     QK 
    \Big( K^{-3 \sigma} (Q L^j)^{1-\sigma}
    +  (Q L^j)^{\frac{7}{8}} \Big).
\end{equation}
Using $K=p^{\varepsilon} Z^{-1} L^j Q^2$ and the definition of $Z$ in \eqref{Zdef} together with the bound for $M'$ in \eqref{eq:M'N'over1plusPbound}, we see that $K^{-3\sigma}(QL^j)^{1-\sigma}\gg (QL^j)^{7/8}$ as long as $Q\ll p^{1-\varepsilon}/L^j\ll p^{1-\varepsilon}/d^2$ and this is satisfied for all $L\ll p^{1/6-\varepsilon}$ using $d^2\ll L^jp^\varepsilon$ from the bound in $N'$ in \eqref{eq:M'N'over1plusPbound}, 
Again using $K =p^{\varepsilon} Z^{-1} L^j Q^2$, we find that the bound \eqref{T''bound} is an increasing function of $Z$.  With the upper bound for $Z$ from \eqref{eq:ZboundUnified}, which gives $K = p^{-1+\varepsilon} Q L^{j/2} d \sqrt{N}$ and the upper bound on $N'$ from \eqref{eq:M'N'over1plusPbound}, we obtain
\begin{align}
    T_{\neq 0}'' 
    &\ll p^\varepsilon L \frac{L^j}{d^2}L \Big(1+ \frac{QpL^{j/2}}{d\sqrt{N}}\Big)\frac{QpL^{j/2}}{d\sqrt{N}}Q K^{1-3\sigma} (QL^j)^{1-\sigma}
    \\&\ll
    p^\varepsilon\frac{L^{2} Q}{p d^2} \frac{Q^2 p^2 L^j}{d^2 N}
     \Big(\frac{Q L^{j/2} d \sqrt{N}}{p}\Big)^{1-3 \sigma} (Q L^j)^{1-\sigma},
\end{align}
which agrees with \eqref{eq:Tneq0''finalbound}.

Next, we analyze the contribution of $T_{\not=0}''$ to $S(N, \mathcal{A}_j)$ via \eqref{eq:SboundInTermsofT} which gives 
\begin{align}
    p^\varepsilon\sum_{d}\frac{\sqrt{N}}{Q\sqrt{L^j}}\Big(\frac{1}{d^2}p^{3\sigma} L^{5j/2+2-5j\sigma/2}N^{-1/2-3\sigma/2}Q^{5-4\sigma}\Big)^{1/2}.
\end{align} Note that this bound, as a function of $Q$, is an increasing function, and so the bound is maximized when $Q = \sqrt{N L^j}/d$.  Moreover, the sum over $d$ produces a convergent series, so it does not contribute to the bound in a substantial way.  
Hence this contributes to $S(N, \mathcal{A}_j)$ via \eqref{eq:SboundInTermsofT} 
the claimed bound
\begin{equation*}
p^\varepsilon L^{-j} 
 p^{3\sigma/2} L^{5j/4+1-5j\sigma/4}N^{-1/4-3\sigma/4}(NL^j)^{5/4-\sigma}
= 
p^\varepsilon N^{\frac{4-7\sigma}{4}} L^{1 + \frac{j}{4}(6-9\sigma)} p^{3\sigma/2}.\qedhere
\end{equation*}
\end{proof}
\begin{lemma}
\label{lemma:Tneq0diag}
The bounds stated in Lemma \ref{lemma:Tneq0''bound} also hold for 
$T_{\neq 0}^{\text{diag}}$ (and its contribution to $S''(N, \mathcal{A}_j)$).
\end{lemma}
\begin{proof}
The proof follows the same way as that in Lemma \ref{lemma:Tneq0''bound}. Using \eqref{eq:Tneq0diagFormula},  \eqref{eq:alphac10c20Fourier}, 
and \eqref{eq:alphaHatFourierTransformL1bound} together with Lemma \ref{lemma:exponentialmodulusell2SeparationofVariables}, we see that the contribution from the terms with $(\ell_2, n_2m)=1$ can be bounded by an expression similar to \eqref{eq:Tneq0''formula2}, where $\ell_1=\ell_2$ and $U(\cdot)$ is replaced by \begin{align}
    U^{\text{diag}}(\cdot) = 
    \sum_{\eta \shortmod{\ell_2}} \frac{\tau(\overline{\eta}) \eta(-\overline{c_0c_{10}' c_{20}'} p n_2 m_0'm')}{\varphi(\ell_2)}
   &\sum_{\substack{(c_1', c_2') = 1 \\ (c_1', \ell_1) = (c_2', \ell_2) = 1 \\ c_1' \asymp \frac{Q}{c_{10}}, c_2' \asymp \frac{C}{c_{20}}}}\sum_{\substack{0<|m|\ll K/c_0m_0'\\ (c_0, m')=1}}
 \chi \psi \eta(c_1') \overline{\chi} \psi \eta(c_2') \\ & \times \overline{\psi}(m')
    e_{c_1'}(\overline{c_0c_{10}' c_{20}' c_2' \ell_1} pm_0'm' \frac{D}{\ell_2})
    w(c_0m_0'm',\cdot),
\end{align}
where $\frac{D}{\ell_2} = n_1 - n_2$ and $K=\frac{L^jQ^2p^\varepsilon}{Z}$. 
Now here we let $q_1 = c_1'$ and $q_2 = c_{0}c_{10}' \ell_2 c_{20}' c_2'$, so that $q_1 \asymp \frac{Q}{c_{0}c_{10}'}$ and $q_2 \asymp Q \ell_2 c_{10}'$.  Here the $q_1$ size, relative to the proof of Lemma \ref{lemma:Tneq0''bound}, is smaller by a factor of $\ell_2$ but without congruence conditions, while $q_2$ has the same size and congruence conditions as before. The conditions for $m'$ are the same as before. Thus the analogue of \eqref{abgbound} remains the same. We apply Theorem \ref{thm:BC} with $a=\frac{D}{\ell_2}$ and $M=\frac{Q}{c_0c_{10}'}$, both of which are smaller than a factor of $\ell_2$ than before. Since the $\eta$-sum contributed to an extra factor of $\ell_2^{1/2}$, we see that the bound for $U^{\text{diag}}(\cdot)$ is worse than that for $U(\cdot)$ by a factor of $\ell_2^{1/2-7/20}=L^{3j/20}$. The estimates for $c_{10}, c_{20}, m_{0}'$ are as before whereas the sum over $\ell_1, \ell_2, n_1, n_2$ is smaller than that in \eqref{ell2n1} by a factor of $L^j$ since $\ell_1=\ell_2$. Therefore, we see that the contribution from $(\ell_2, n_2m)=1$ to $T^{\text{diag}}_{\not=0}$ is smaller than that of $T_{\not=0}$ by a factor of $L^{3j/20-1}$. The comment after Lemma \ref{lemma:exponentialmodulusell2SeparationofVariables} shows that the contribution from $(\ell, n_2m)>1$ is even smaller. 
\end{proof}
\subsection{Conclusion}
Now we put everything together.  From Lemmas \ref{lemma:conclusionBeforePoisson}, \ref{lemma:T0'bound}, \ref{lemma:Tneq0''bound} and \ref{lemma:Tneq0diag}, we obtain for $L<p^{1/6-\varepsilon}$, and with $\sigma = 1/20$,
\begin{align}
    S(N, \mathcal{A}_j) \ll&\ p^\varepsilon\Bigg(
    \sqrt{NLp} + N^{1-7\sigma/4} p^{3\sigma/2}L^{3j/2+1-9j\sigma/4}
    +
    \frac{N L^{1+j}}{\sqrt{p}}  + N^{3/4} L^{1+j} +
     p^{1/2}N^{1/4} L^{\frac12 + \frac{3j}{4}} \Bigg).
\end{align}
When $L\ll N^{1/10}$, the last three terms above may be absorbed by the term $\sqrt{NLp}$, so it simplifies as 
\begin{align}
    S(N,\mathcal{A}_j) \ll p^\varepsilon \Big( \sqrt{NLp}+N^{1-7\sigma/4} p^{3\sigma/2}L^{3j/2+1-9j\sigma/4}\Big).
\end{align}
This concludes the proof of Proposition \ref{thm:SNAbound} and hence Theorem \ref{thm:mainthm}.
\begin{remark}
    If $\chi$ is trivial, then one can apply \cite[Theorem 1.8]{lichtman} to obtain better bounds for $T_{\not=0}''$ when $c_{10}, c_{20}$ are small. This would lead to an improvement of Theorem \ref{thm:mainthm} when $\chi$ is trivial, however it appears the resulting bound is inferior to the result of \cite{Zacharias}.
\end{remark}

\section{Adjustments for Maass forms and Eisenstein series}\label{sect:adjustment}

When $f$ and/or $g$ is a Maass form, the preliminaries in Section \ref{sect:prelim} and the whole argument in Section \ref{sect:proof} goes through with the following minor adjustments.
\begin{enumerate}
    \item The gamma factors in the functional equation \eqref{eq:FE} have to be adjusted accordingly with \cite[Chp. 5.11]{IK04}. We obtain the same conclusion \eqref{eq:AFEdyadicVersion} and \eqref{eq:exp sum} with an appropriate weight function $w_N$.
    \item The Voronoi summation formula stated in Lemma \ref{prop:voronoi} has to be adjusted accordingly with \cite[Theorem A.4]{KMV02}. In the application of Voronoi summation in Section \ref{section:FunctionalEquations} and \ref{sect:BoundingS0}, the Bessel functions have to be adjusted accordingly with $\pm$ sign introduced on each of the $m$ and $n$ sums. The rest of the argument goes through as \eqref{JBesAs} also holds for the other Bessel functions with appropriate weight functions $W_k$.
    \item The bound we used for the Fourier coefficients $\lambda_f(m)$ and $\lambda_g(n)$ are in terms of \eqref{RS} and \eqref{BoundAtp}, which carry through unchanged.
\end{enumerate}

Now we turn our attention to the case when $g$ is an Eisenstein series. The only 
additional
adjustment to make is with the Voronoi summation formula. With $\lambda_g(n)=\tau(n)$, the Voronoi summation applied on the $n$-sum in Sections \ref{section:FunctionalEquations} and Section \ref{sect:BoundingS0} creates a main term according to \cite[Theorem A.4]{KMV02}. Other than the main term, the rest of the argument goes through as discussed above.

The contribution of the main term to $S'(N,\mathcal{A}_j)$ in Section \ref{section:FunctionalEquations} is given by 
  \begin{equation}\label{main term}
        E'(N, \mathcal{A}_j):=\frac{1}{C \sqrt{p}}\sum_{\ell\in\mathcal{L}_j}b_j(\ell)
    \sum_{d} 
    \sum_{\substack{(c, \ell) = 1 }}  \frac{\chi(c) }{c^3 d}      \, 
    \sum_{m } \, 
    \, \overline{\lambda_{f}}(m)  
           S(0, - m \overline{p}; c)
           I_{\mathrm{Eis}}(m,\cdot),
    \end{equation}
    where assuming $f$ is holomorphic (say) we have
    \begin{equation}
          I_{\mathrm{Eis}}(m, \, \cdot\, )= 
 \int_{0}^{\infty} \, \int_{0}^{\infty} \ w_{N\ell}(x)w_{N}(y) U\left( \frac{x-y\ell}{cdC}\right) J_{k-1}\left(\frac{4\pi}{c}\sqrt{\frac{mx}{p}}\right) \log\left( \frac{y e^{2 \gamma}}{c^2} \right) 
        w(\cdot)
        \ dx\, dy.   
    \end{equation}
Changing variables $y \rightarrow y + \frac{x}{\ell}$ leads to an $x$-integral with 
only the single phase coming from the Bessel function, which means that we are only in the non-oscillatory range.  Thus we have
 $M' \ll \frac{Q^2 p^{1+\varepsilon}}{N \ell}$ and $I_{\text{Eis}} \ll Q CN$.  The trivial bound on $E'$ from this is then
\begin{equation}
    \ll \frac{p^{\varepsilon}}{C \sqrt{p}} L \frac{1}{Q^2} QCN M' 
    \ll \frac{p^{\varepsilon}}{ \sqrt{p}} L \frac{N}{Q} \frac{Q^2 p}{N \ell} 
    = p^{\varepsilon} \sqrt{p} Q \frac{L}{\ell} \ll p^{\varepsilon}  \sqrt{pN \ell} \frac{L}{\ell} \ll p^{\varepsilon} \sqrt{pN L}.
\end{equation}
Similarly, the main term contribution to $S_0(N,\mathcal{A}_j)$ in Section \ref{sect:BoundingS0} can be bounded by the same bound, and they are subsumed by the bound of $T^{(0)}(N,\mathcal{A}_j)$ in Lemma \ref{lemma:Tosc0bound}.

\bibliographystyle{alpha}
\bibliography{ref}	

\end{document}